\documentclass{siamltex}
\usepackage[T1]{fontenc}
\usepackage[latin9]{inputenc}
\usepackage{amssymb}
\usepackage{sw20siam}
\usepackage{mathrsfs}
\usepackage{graphicx}
\usepackage{mathtools}
\usepackage{multirow}
\usepackage{cite}

\everymath{\displaystyle}
\makeatletter
\numberwithin{equation}{section}
\numberwithin{figure}{section}

\newcommand*\md{\mathop{}\!\mathrm{d}}

\usepackage{color}
\usepackage[colorlinks]{hyperref}

\newtheorem{thm}[theorem]{Theorem}
\newtheorem{lem}[theorem]{Lemma}
\newtheorem{prop}[theorem]{Proposition}
\newtheorem{cor}[theorem]{Corollary}
\newtheorem{rem}[theorem]{Remark}

\newtheorem{condition}[theorem]{Assumption}
\newtheorem{example}[theorem]{Example}
\makeatletter
\def\@thm#1#2{%
  \stepcounter{#1}%
  \protected@edef\@currentlabel{\csname p@#1\endcsname\csname the#1\endcsname}%
  \@ifnextchar[{\@ythm{#1}{#2}}{\@xthm{#1}{#2}}}
\def\@begintheorem#1#2{\par\phantomsection\bgroup{\scshape #1\ #2. }\it\ignorespaces}
\def\@opargbegintheorem#1#2#3{\par\phantomsection\bgroup%
   {\scshape #1\ #2\ ({\upshape #3}). }\it\ignorespaces}
\makeatother

\newcommand{\R}{\mathbb{R}}

\newcommand{\E}{\mathbb{E}}
\newcommand{\eps}{\varepsilon}

\makeatother

\title{Optimal Control of SDEs with Merely Measurable Drift: An HJB Approach\thanks{{\bf Funding:} {This work was financially supported by the National Key R\&D Program of China  (No. 2023YFA1009002).
The first author was supported by the National Natural Science Foundation of China (12222103, 42450269), and LMNS at Fudan University.
The second author was supported by the National Natural Science Foundation of China
(No. 12371443), the Natural Science Foundation of Jilin Province,
and the Changbai Talent Program of Jilin Province.}}}

\author{
  Kai Du%
  \thanks{Shanghai Center of Mathematical Sciences, Fudan University, Shanghai 200433, China
          (Email: {\tt kdu@fudan.edu.cn}).}%
  \and
  Qingmeng Wei%
  \thanks{Corresponding author. School of Mathematics and Statistics, Northeast Normal University, Changchun 130024, China
          (Email: {\tt weiqm100@nenu.edu.cn}).}%
}

\begin{document}
\maketitle

\begin{abstract}
We investigate a finite-horizon optimal control problem for a diffusion with additive noise, where the control-dependent drift and running cost are merely measurable in the state variable and satisfy a global $L^p$-integrability condition.
Such low regularity rules out the direct use of the classic Pontryagin maximum principle and also invalidates the standard proof of the Bellman principle of optimality.
We address these difficulties by analyzing the associated Hamilton--Jacobi--Bellman (HJB) equation.
Working in a weak formulation of admissible controls, we first establish the Krylov estimates needed to make the control problem well defined.
Using PDE techniques together with a policy iteration scheme, we prove that the HJB equation admits a unique Sobolev solution in $W^{1,2}_p$, and this solution coincides with the value function of the control problem.
Based on this identification, we establish a verification theorem and recover the Bellman optimality principle without imposing any additional smoothness assumptions.

We further investigate a mollification scheme depending on a parameter $\varepsilon > 0$.
It turns out that the smoothed value functions $V_{\varepsilon}$ may fail to converge to the original value function $V$ as $\varepsilon \to 0$, and we provide an explicit counterexample.
To resolve this, we identify a structural condition on the control set.
When the control set is countable, convergence $V_{\varepsilon} \to V$ holds locally uniformly.
\end{abstract}
\keyphrases{measurable drift, optimal control,  HJB equation, dynamic programming principle}
\AMclass{93E20, 35Q93}

\section{Introduction}

In this paper we study the optimal control problem for the following stochastic differential equation (SDE)
\begin{subequations}\label{eq:oc_problem}
\begin{equation}
\md X_{t}=b(t,X_{t},\alpha_{t})\md t+\sqrt{2}\md W_{t},\quad X_{s}=x\in\mathbb{R}^{d}{.}\label{eq:state}
\end{equation}
{For an admissible weak control system $\mathbb A$ with control $\alpha$ and state $X$, the cost is}
\begin{equation}
J({\mathbb A};s,x)=\mathbb{E}^{{\mathbb A}}\int_{s}^{T}f(t,X_{t},\alpha_{t})\md t.\label{eq:cost}
\end{equation}
The value function is
\begin{equation}\label{eq:value}
  V(s,x)
  \;=\;
  \inf_{{\mathbb A\in\mathfrak A_{s,T}(x)}}
  J({\mathbb A};s,x),\quad (s,x)\in [0,T]\times \mathbb R^d.
\end{equation}
{The class $\mathfrak A_{s,T}(x)$ of admissible weak control systems is defined below.}
\end{subequations}
A distinctive feature of this work is that the drift coefficient
$b$ and the running cost $f$ are assumed
to be only measurable in the state variable---no smoothness or continuity is required.
To the best of our knowledge, this situation has not been systematically
studied in the HJB--Sobolev framework.

SDEs with irregular drift coefficients model phenomena such as sudden switches
or singular forces; see \cite{kloeden1995numerical}.
In such settings, the classical Lipschitz-based diffusion theory fails.
Still, by leveraging the Zvonkin transform and Krylov's estimate,
Veretennikov~\cite{Veretennikov1981} gave the strong existence and uniqueness
when the drift was bounded and measurable.
Subsequent studies have extended these results to wider classes of drift coefficients (see e.g. \cite{KrylovRockner2005,Zhang2005,fedrizzi2011pathwise,Zhang2011,Xia2020}).
 These results highlight both the difficulties and the potential of studying
optimal control problems with irregular coefficients, which we want to address in the present paper.

One concrete motivation comes from the controlled transport of particles in a
fluid.  In a Lagrangian description, a diffusive tracer particle moving in a
velocity field \(\mathbf{u}\) is naturally modeled by
\[
\md X_t=\mathbf{u}(t,X_t)\,\md t+\sqrt2\,\md W_t .
\]
Such stochastic Lagrangian descriptions are standard in fluid mechanics and
turbulent dispersion; see, for instance,
\cite{ConstantinIyer2008,Pope1994}. In many control problems, the control acts
by changing parameters of the ambient flow field, rather than by directly
forcing each particle. For example, in mixing problems the velocity field may be
generated by time-modulated force fields or other external inputs
\cite{MathewEtAl2007}. From this viewpoint, \(b(t,x,a)\) represents a controlled
velocity field. Since shear flows, including Couette-type profiles, form a basic
class in Euler dynamics \cite{BedrossianMasmoudi2015}, controlled or
coarse-grained descriptions of such flows may naturally lead to drift
coefficients with weak spatial regularity. This provides a physical motivation
for studying control-dependent merely measurable drifts.

The problem poses key difficulties due to the lack of regularity in the drift.
First, without differentiability in the state, {the classical Pontryagin maximum principle}
cannot be applied in its standard form, since one cannot directly form the
usual adjoint equations.
Second, the dynamic programming approach---based on the Bellman
optimality principle and viscosity solution theory---is also not directly
applicable, as these arguments typically require suitable continuity or
Lipschitz conditions on the coefficients, which are not available in our
setting; see, e.g., \cite{YongZhou1999}.  This motivates us to seek a route which does not rely directly on either the
standard maximum principle or the classical dynamic programming principle.

 The present paper advances the HJB analysis of stochastic
control problems with irregular coefficients by treating finite-horizon
control-dependent merely measurable drifts under a global $L^p$-integrability
condition.
Our approach circumvents the conventional reliance on the dynamic programming principle and instead proceeds in the opposite direction.
We begin by formulating the following HJB equation associated with the control problem:
\begin{equation}\label{eq:hjb}
    \partial_{s}u +\Delta u+ \inf_{a\in A} \big\{b(s,x,a)\cdot\nabla u+f(s,x,a) \big\}=0,
    \quad u(T,x)=0,
\end{equation}
and show that it admits a unique Sobolev solution under our assumptions, and the value function $V(\cdot,\cdot)$ of the control problem coincides with this solution (Theorem~\ref{Thm:exi-uni}).
The core of our proof relies on a policy iteration argument, which allows us to connect the control problem to the HJB equation directly, despite the irregularity of the coefficients.
With this identification in hand, the Bellman optimality principle follows as a corollary (Theorem~\ref{thm:Bellman}).
This fact may be useful for constructing discrete-time approximations of the control problem, which we plan to investigate in future work.
Finally, the identification of the value function and HJB solution yields a verification theorem for near-optimal feedback laws (Theorem~\ref{thm:verification}).
In particular, although the value function is defined over weak open-loop
control systems, it coincides with the value obtained by restricting to Borel
feedback controls (Corollary~\ref{cor:feedback-value}).

In the second part of the paper, we examine a mollification-based approximation scheme.
A well-known approach for solving SDEs with irregular drift is to smooth the drift coefficient, solve the regularized equation, and then pass to the limit as the smoothing parameter \(\varepsilon\to 0\).
This naturally raises the question whether
the same idea works for our optimal control problem.
Denoting by \(V_{\varepsilon}\) the value functions of
the mollified control problems, we find
the convergence behavior to be \emph{subtle}: the smoothed value functions \(V_{\varepsilon}\) do not necessarily converge to the original value function $V$, as $\varepsilon\to 0$.  An explicit counterexample is constructed to demonstrate this (Example~\ref{ex:failure-Lp}).  The same construction also exhibits a bounded-coefficient limiting case outside the global \(L^p\) framework.
We prove in Proposition~\ref{thm:liminf} that
\(
\liminf_{\varepsilon\to 0} V_{\varepsilon} \ge V,
\)
yet the opposite inequality may fail without additional structure.
Interestingly, we find that if the control set
is countable, then \(\limsup_{\varepsilon\to 0} V_{\varepsilon} \le V\) (Proposition~\ref{Prop:V>V_n}).
Thus, under this structural condition, the classical smoothing approach remains reliable for the
analysis and computation.

In recent years, several related works have addressed stochastic control
problems with nonsmooth coefficients from complementary viewpoints; see, for
example, \cite{MenoukeuPamen2023,DeFeo2025,Criens2025,BogsoEtAl2025}.
Menoukeu-Pamen and Tangpi
\cite{MenoukeuPamen2023} established a stochastic maximum principle for
controlled SDEs with measurable drifts. In their framework, the drift is
decomposed as
$
b(t,x,a)=b_1(t,x)+b_2(t,x,a),
$
where the singular component $b_1$ is bounded, Borel measurable, and
independent of the control, while the control-dependent component $b_2$
satisfies additional smoothness assumptions. More recently, Bogso et al.
\cite{BogsoEtAl2025}
studied a stochastic maximum principle for controlled diffusion systems with
drifts of the form
$
b(t,x,a)=b_1(t,x)+b_2(x)b_3(t,a),
$
where $b_1$ is bounded and Borel measurable, $b_2$ is of bounded variation in
the state variable, and $b_3$ is bounded and smooth. Their proof relies on
Malliavin--Sobolev differentiability of the state process, an explicit
representation of the first variation in terms of space--time local time, and a
smooth approximation argument.

Another closely related line is PDE-based. De Feo \cite{DeFeo2025} studies
infinite-horizon discounted, fully nonlinear stochastic control problems with
time-homogeneous measurable coefficients and locally uniformly elliptic
diffusion by means of $L^p$-viscosity solutions. In that framework, the
solvability of the state equation for admissible controls is incorporated into
the admissibility requirement, while closed-loop solvability and uniqueness of
the HJB solution are obtained under additional boundedness, ellipticity, and
selection assumptions through verification arguments. We also mention the
recent semigroup approach of Criens \cite{Criens2025}, which treats controlled
SDEs with measurable coefficients, uniformly elliptic diffusion, and an
$L_d$-drift, and relates the value functions to $L^p$-viscosity solutions.

 The present paper is complementary to these works both in method and in
the nature of the assumptions. Unlike the maximum-principle approaches
in \cite{MenoukeuPamen2023,BogsoEtAl2025}, we do not derive Pontryagin-type necessary conditions,
but instead characterize the value function through a Sobolev solution
of the associated HJB equation. Compared with the PDE and
$L^p$-viscosity approaches in \cite{DeFeo2025,Criens2025}, which allow rough,
uniformly elliptic, and possibly state- and control-dependent diffusion
coefficients, we consider finite-horizon, time-inhomogeneous problems
with additive nondegenerate noise. Consequently, the second-order term
in our HJB equation is linear, whereas a control-dependent diffusion
generally leads to a fully nonlinear HJB equation.
 The additive-noise structure allows us to work in a single
global unweighted $L^p(Q_T)$ Sobolev framework,
{to obtain strong existence and uniqueness
for the state equation under Borel feedback laws and a uniform
Krylov estimate for all admissible weak control systems},
and to identify the value function with the unique solution
in $W_p^{1,2}(Q_T)$.

The global $L^p(Q_T)$ assumption in the present paper is not a strict
generalization of the bounded measurable-drift setting considered in
\cite{MenoukeuPamen2023}. Indeed, since
$Q_T=[0,T]\times\mathbb{R}^d$ has infinite Lebesgue measure, a bounded
measurable function on $Q_T$ need not belong to $L^p(Q_T)$. Conversely,
our assumption permits certain unbounded coefficients with sufficient
global $L^p$ integrability. Thus, the two frameworks are complementary
and, in general, incomparable.

A direct extension of the present results to the infinite-horizon
discounted setting considered in \cite{DeFeo2025} is not immediate,
since two distinct difficulties arise. First, extending our additive-noise
framework to measurable, possibly control-dependent diffusion
coefficients would require substantially different probabilistic and
analytic tools. At the probabilistic level, one would need either
well-posedness and Krylov-type estimates that are uniform over the
admissible controls, or a weak or relaxed formulation together with an
appropriate Markov-selection principle. At the PDE level, the control
dependence of the diffusion coefficient leads to a fully nonlinear HJB
equation, for which suitable Sobolev regularity, comparison, and
uniqueness results with measurable second-order coefficients would be
needed.

Second, the passage from a finite to an infinite time horizon cannot be
justified under Assumption~\ref{ass:main} alone. The estimates in the present paper
are obtained on each fixed cylinder $Q_T$, and Assumption~\ref{ass:main} does not
provide the uniform-in-horizon bounds needed to pass to the limit as
$T\to\infty$. An infinite-horizon extension would therefore require
well-posedness of the controlled state equation on the infinite time
interval, integrability of the discounted cost, and Krylov and Sobolev
estimates that remain uniform with respect to the horizon. In the
time-homogeneous case, the finite-horizon terminal-value problem would
be replaced by a stationary discounted HJB equation. Its analysis would
require an appropriate global solution class and, when unbounded
solutions are allowed, suitable growth conditions at spatial infinity.
These issues require additional assumptions and new arguments, and constitute an interesting direction for future research.

Accordingly, we assume that $|b|+|f|$ is bounded uniformly over the actions
by a common $L^p(Q_T)$ envelope with $p>d+2$,
which is compatible with both the well-posedness theory for irregular SDEs and
the parabolic Sobolev estimates used in our analysis. Extensions to weaker
integrability frameworks, locally integrable coefficients with controlled
growth, or more general diffusion matrices would require sharper analytical
estimates and are left for our future investigation.

We next introduce the notation and some preliminary results, and then give the
precise formulation of the controlled problem.

\section*{Preliminaries}

Fix $T>0$ and set $Q_T=[0,T]\times\mathbb R^d$. Test functions in the
distributional formulations below have compact support in $(0,T)\times\mathbb R^d$.
A stochastic basis below means
a complete filtered probability space
$(\Omega,\mathcal F, \{\mathcal F_t\}_{t\in[0,T]},\mathbb P)$
satisfying the usual conditions and supporting a $d$-dimensional Brownian
motion $W$.

{We first record a Krylov estimate for existing solutions and
strong solvability under Borel feedback controls.}

\begin{lem}[Solvability and Krylov estimate for state equations]
\label{lem:state-equation-solvability}
Let $E$ be a standard Borel space, let $B:Q_T\times E\to\mathbb R^d$ be Borel
measurable, and suppose that
$$
|B(t,x,e)|\le \Psi(t,x),\qquad (t,x,e)\in Q_T\times E,
$$
where $\Psi\in L^p(Q_T)$ with $p>d+2$. Fix $(s,x)\in Q_T$.

{\rm(i)} {Let $\beta=\{\beta_t\}_{t\in[s,T]}$ be an $E$-valued
progressively measurable process on a stochastic basis supporting $W$.
Let $X$ be an adapted continuous process solving}
$$
\md X_t=B(t,X_t,\beta_t)\,\md t+\sqrt{2}\,\md W_t,
\qquad X_s=x,\quad t\in[s,T]{.}
$$
{Then}, for every $h\in L^p(Q_T)$, {$X$} satisfies the Krylov
estimate
\begin{equation}\label{eq:krylov-weak}
\mathbb E\int_s^T |h(t,X_t)|\,\md t
\le C\|h\|_{L^p} .
\end{equation}

{\rm(ii)} If $\theta:Q_T\to E$ is Borel measurable, then the closed-loop SDE
$$
\md X_t=B(t,X_t,\theta(t,X_t))\,\md t+\sqrt{2}\,\md W_t,
\qquad X_s=x,\quad t\in[s,T],
$$
admits a unique strong solution. This solution also
satisfies \eqref{eq:krylov-weak}.

In both cases the constant $C$ depends only on $T,d,p$ and
$\|\Psi\|_{L^p}$.
\end{lem}

{
The proof in Appendix~\ref{app:state-details} obtains the estimate in {\rm(i)}
by localizing along the given solution. In {\rm(ii)}, the Borel feedback
yields a deterministic drift.}

\begin{rem}[{Failure of weak existence or uniqueness for prescribed open-loop controls}]\label{Exam:1.2}
{ Even for bounded drifts, a prescribed open-loop control can give rise to
solutions with different laws or to no weak solution at all.
\hyperlink{app:state-examples}{Appendix~\ref*{app:state-details}} gives examples of both phenomena with
the same prescribed relation $\beta_t=\sqrt{2}\,W_t$.}
\end{rem}

\subsection*{Problem Formulation}\label{sec:setting}

 Let $A$ be a nonempty Borel subset of a Polish space. This choice ensures
that $A$ has a standard Borel structure, which is needed for the measurable
selection arguments used below.
Throughout the paper we impose the following standing assumption.

\begin{condition}\label{ass:main}
Fix \(p>d+2\).
The mapping
\((b,f):{Q_T}\times A
        \to\mathbb{R}^{d}\times\mathbb{R}\)
is Borel measurable, and there exists a Borel function
\(\Phi\in L^{p}({Q_T})\) such that, for all $(t,x,a)\in {Q_T}\times A$,
$$|b(t,x,a)| + |f(t,x,a)|\leq \Phi(t,x).$$
\end{condition}

 We use a weak formulation for the control problem. For
\(0\le s\le r\le T\) and \(x\in\mathbb R^d\), an admissible weak control system on
\([s,r]\) starting from \(x\) is a tuple
\[
\mathbb A=(\Omega,\mathcal F,\{\mathcal F_t\}_{t\in[s,r]},\mathbb P,W,\alpha,X)
\]
such that the stochastic basis satisfies the usual conditions, \(W\) is a
\(d\)-dimensional Brownian motion, \(\alpha\) is an \(A\)-valued progressively
measurable process, and \(X\) is an adapted continuous process satisfying
\[
X_t=x+\int_s^t b(\tau,X_\tau,\alpha_\tau)\,\md \tau+\sqrt2\,(W_t-W_s),
\qquad t\in[s,r],
\]
in the weak sense. We denote this class by \(\mathfrak A_{s,r}(x)\). It is
nonempty by Lemma~\ref{lem:state-equation-solvability}{{\rm(ii)}, applied to a constant feedback $\theta\equiv a_0$ with $a_0\in A$. By part {\rm(i)},} every
\(\mathbb A\in\mathfrak A_{s,r}(x)\) satisfies the Krylov estimate
\begin{equation}\label{eq:krylov-admissible}
\mathbb E^{\mathbb A}\int_s^r |h(t,X_t)|\,\md t
\le C\|h\|_{L^p },\quad h\in L^p(Q_T),
\end{equation}
with \(C\) depending only on \(T,d,p\) and \(\|\Phi\|_{L^p}\).

\smallskip

For \(\mathbb A\in\mathfrak A_{s,T}(x)\), define
\[
J(\mathbb A;s,x):=
\mathbb E^{\mathbb A}\int_s^T f(t,X_t,\alpha_t)\,\md t .
\]
Then Assumption~\ref{ass:main} and \eqref{eq:krylov-admissible} give
\begin{equation}\label{eq:value_bound}
|J(\mathbb A;s,x)|
\le
\mathbb E^{\mathbb A}\int_s^T \Phi(t,X_t)\,\md t
\le
C\|\Phi\|_{L^p}.
\end{equation}
Thus the value function
\[
V(s,x):=\inf_{\mathbb A\in\mathfrak A_{s,T}(x)}J(\mathbb A;s,x)
\]
is finite, with $V(T,x)=0$ since the cost on a singleton interval is zero, and the optimal control problem \eqref{eq:oc_problem} is well
defined.

\smallskip

\begin{rem}
{ Remark~\ref{Exam:1.2} illustrates why the state $X$ is part of the admissible system. A prescribed open-loop control may admit several state solutions or none, while each admissible pair $(\alpha,X)$ on its given stochastic basis has a well-defined cost.}
When the control is of feedback
form, \(\alpha_t=\theta(t,X_t)\) with a Borel map \(\theta:Q_T\to A\), the
closed-loop drift
\[
B^\theta(t,x):=b(t,x,\theta(t,x))
\]
is deterministic and satisfies \(|B^\theta(t,x)|\le\Phi(t,x)\). Hence
Lemma~\ref{lem:state-equation-solvability}{\rm(ii)} gives a unique strong
solution on any prescribed stochastic basis. Such a feedback control therefore
induces an admissible weak control system; its cost will be denoted simply by
\(J(\theta;s,x)\), or by \(J(\alpha;s,x)\) when the feedback process is clear
from the context.
\end{rem}

The flat torus $\mathbb T^d$ gives the corresponding periodic formulation
when the coefficients are periodic in space.
\smallskip

 The remainder of the paper is structured as follows.
In Section~\ref{sec:HJB}, we prove that the value function \(V\) is the unique
Sobolev solution of the HJB equation~\eqref{eq:hjb} and establish the dynamic
programming principle, together with a verification theorem.
Section~\ref{sec:mollify} analyzes a mollification scheme and studies the
related convergence properties.
Appendix~A contains a whole-space weak maximum principle and the corresponding
weak comparison principle for the HJB equation. {Appendix~\ref{app:state-details}
contains the proof of Lemma~\ref{lem:state-equation-solvability} and the examples
illustrating Remark~\ref{Exam:1.2}.}

\section{Value Function, HJB Equation, and Bellman  Principle}\label{sec:HJB}

Having established the well-posedness of the control problem, we now turn to its analytical structure.
Our first goal is to characterize the value
function \(V\) as the unique Sobolev solution of the HJB equation~\eqref{eq:hjb}.
Building on this characterization, we then recover the Bellman principle of optimality (i.e., the dynamic programming principle) and prove a verification theorem that yields near-optimal feedback controls.

\subsection{Main Results}\label{subsec:hjb-main}

Our first theorem shows that the value function is precisely the (strong) solution of the HJB equation and that this solution is unique in the natural Sobolev class.
For this, we recall the standard parabolic Sobolev space $W_p^{1,2}(Q_T)$, endowed with the
norm
$$
\|u\|_{W_p^{1,2}}
:=
\|u\|_{L^p}
+\|\partial_t u\|_{L^p}
+\|\nabla u\|_{L^p}
+\|D^2 u\|_{L^p} .
$$
For $\alpha\in(0,1)$, the notation
$C_{t,x}^{(1+\alpha)/2,1+\alpha}(Q_T)$ denotes the parabolic H\"older space
with exponent $(1+\alpha)/2$ in time and $1+\alpha$ in space. Its norm may be
taken as
$$
\begin{aligned}
\|u\|_{C_{t,x}^{(1+\alpha)/2,1+\alpha} }
:={}&
\|u\|_{L^\infty }
+\|\nabla u\|_{L^\infty }
+
\sup_{\substack{t\ne s\\ x\in\mathbb R^d}}
\frac{|u(t,x)-u(s,x)|}{|t-s|^{(1+\alpha)/2}}
\\
&+
\sup_{\substack{(t,x)\ne(s,y)}}
\frac{|\nabla u(t,x)-\nabla u(s,y)|}
{|t-s|^{\alpha/2}+|x-y|^\alpha}.
\end{aligned}
$$

\begin{theorem}\label{Thm:exi-uni}
Under Assumption \ref{ass:main}, it holds that
\begin{enumerate}
\item[\rm i)] HJB equation (\ref{eq:hjb}) admits a unique Sobolev solution $u\in W_p^{1,2}(Q_T)$.
\item[\rm ii)] The value function $V$ coincides with the continuous version of the solution $u$ of HJB equation (\ref{eq:hjb}).
\end{enumerate}
\end{theorem}

\smallskip

The proof of Theorem~\ref{Thm:exi-uni} is given in the next subsection.
Let us give a comment on the continuity of $u$ and $V$.
\smallskip

\begin{rem}\label{rem:sobolev-ito}
Since \(p>d+2\), the Sobolev embedding
\(W^{1,2}_{p}(Q_T)\hookrightarrow C^{\delta}(Q_T)\) ensures that every
\(W^{1,2}_{p}\)-function admits a continuous modification.
Throughout the paper we always work with this continuous version.
The generalized It\^o formula applies along every admissible system by
smooth approximation in $W_p^{1,2}$, as in the proof of
\cite[Theorem~3.7]{KrylovRockner2005}. Indeed, \eqref{eq:krylov-admissible}
controls the errors in the time derivative and Laplacian. The gradients
converge uniformly by Sobolev embedding, which gives convergence of the
drift term in $L^1$ since $\mathbb E\int_s^r\Phi(t,X_t)\,\md t<\infty$ on each admissible interval $[s,r]$,
and of the stochastic integral by It\^o's isometry.
In the proof of Theorem~\ref{Thm:exi-uni}, we identify the value function
\(V\) pointwise with the continuous modification of the
solution \(u\) to the HJB equation.
Consequently, \(V\) inherits the regularity properties of \(u\).
\end{rem}
\smallskip

 \begin{rem}[Nonzero terminal cost]
  A nonzero terminal cost in \eqref{eq:hjb} can also be incorporated within the same
Sobolev framework under a   trace condition. More precisely, if the
cost functional contains an additional terminal term
$\psi(X_T^{s,x;\alpha})$, then the corresponding HJB equation is equipped with
the terminal condition
$
u(T,x)=\psi(x), \ x\in\mathbb R^d .
$
A natural condition on the terminal datum is
$$
\psi\in W_p^{2-2/p}(\mathbb R^d),\qquad p>d+2,
$$
which is the terminal trace space associated with the parabolic Sobolev space
$W_p^{1,2}(Q_T)$.   Moreover, since
$p>d+2$, the Sobolev--Morrey embedding yields
$
W_p^{2-2/p}(\mathbb R^d)\hookrightarrow C_b(\mathbb R^d).
$
Thus $\psi$ admits a bounded continuous representative. Consequently, for every
\(\mathbb A\in\mathfrak A_{s,T}(x)\),
\[
\mathbb E^{\mathbb A}|\psi(X_T)|<\infty .
\]

If $\psi$ is only bounded and measurable, it is not directly covered by the
Sobolev trace framework. Treating such terminal costs would require additional
arguments, for instance through suitable approximation and stability results,
or through an $L^p$-viscosity solution approach. This will be an interesting
direction for future research.
\end{rem}

 \smallskip

In this low-regularity setting, the Bellman principle of optimality follows \emph{as a corollary} of the preceding theorem rather than serving as its starting assumption.
\begin{thm}[Principle of optimality]\label{thm:Bellman}
Under Assumption~\ref{ass:main}, for all $(s,x)\in Q_T$ and $t\in[s,T]$,
\begin{equation}\label{eq:DPP}
V(s,x)=\inf_{\mathbb A\in\mathfrak A_{s,t}(x)}\,
\mathbb E^{\mathbb A}\Big[\int_s^{t} f(r,X_r,\alpha_r)\,\mathrm dr
+ V\big(t,X_t\big)\Big],
\end{equation}
where \((\alpha,X)\) denotes the control-state pair in the admissible system
\(\mathbb A\).
\end{thm}

\begin{proof}
By Theorem~\ref{Thm:exi-uni}, the value function $V\in W^{1,2}_p(Q_T)$ solves the HJB equation \eqref{eq:hjb} and admits a continuous modification. Fix $t\in[s,T]$.

{(``$\le$'')} Let \(\mathbb A\in\mathfrak A_{s,t}(x)\). By the Krylov estimate \eqref{eq:krylov-admissible}, the generalized It\^o formula
(see Remark~\ref{rem:sobolev-ito}) is applicable to
$V(r,X_r)$. Thus
\[
\begin{aligned}
V\big(t,X_t\big)
&=V(s,x)
+\!\int_s^{t}\!\sqrt{2}\,\nabla V\big(r,X_r\big)\,\mathrm dW_r \\
&\quad+\!\int_s^{t}\!\big[(\partial_r V+\Delta V)\big(r,X_r\big)
+b\big(r,X_r,\alpha_r\big)\!\cdot\!\nabla V\big(r,X_r\big)\big]\,\mathrm dr \\
&\ge V(s,x) - \!\int_s^{t}\! f\big(r,X_r,\alpha_r\big)\,\mathrm dr
+\!\int_s^{t}\!\sqrt{2}\,\nabla V\big(r,X_r\big)\,\mathrm dW_r,
\end{aligned}
\]
where the inequality uses that $V$ solves the HJB equation.
Taking expectations (the stochastic integral has mean $0$) yields
\[
V(s,x)\le \mathbb E^{\mathbb A}\Big[\int_s^{t} f(r,X_r,\alpha_r)\,\mathrm dr
+ V\big(t,X_t\big)\Big].
\]
Now take the infimum over \(\mathbb A\in\mathfrak A_{s,t}(x)\) to obtain the
``$\le$'' direction of \eqref{eq:DPP}.

\smallskip

{(``$\ge$'')} By the measurable selection argument given below, for each $\epsilon>0$ there exists a Borel selector
\(\alpha^\epsilon:[s,t]\times\R^d\to A\) such that, for a.e.\ $(r,y)\in[s,t]\times\R^d$,
\begin{equation}\label{inequ: bf-leq-1}
\begin{aligned}
& b(r,y,\alpha^\epsilon(r,y))\cdot\nabla V(r,y)+f(r,y,\alpha^\epsilon(r,y))\\
& \le \inf_{a\in A}\!\big\{b(r,y,a)\cdot\nabla V(r,y)+f(r,y,a)\big\}+C\epsilon,
\end{aligned}
\end{equation}
with $C$ independent of $\epsilon$.
Hence, by Lemma~\ref{lem:state-equation-solvability}{\rm(ii)}, the closed-loop equation
$$
\md X_r^\epsilon
=b(r,X_r^\epsilon,\alpha^\epsilon(r,X_r^\epsilon))\,\md r
+\sqrt2\,\md W_r,\qquad X_s^\epsilon=x,
$$
admits a unique strong solution and satisfies the Krylov estimate. Then
the feedback control
$
\alpha_r^\varepsilon:=\alpha^\varepsilon(r,X_r^\varepsilon),
$ $r\in[s,t],
$
induces an admissible weak control system in \(\mathfrak A_{s,t}(x)\).

Applying the generalized It\^o formula to $V(r,X_r^\epsilon)$ and using \eqref{inequ: bf-leq-1} together with the HJB equation gives
\[
\begin{aligned}
V\big(t,X_t^\epsilon\big)
&\le V(s,x) - \!\int_s^{t}\!\!\big[f\big(r,X_r^\epsilon,\alpha_r^\epsilon\big)-C\epsilon\big]\,\mathrm dr\\
& \quad +\!\int_s^{t}\!\sqrt{2}\,\nabla V\big(r,X_r^\epsilon\big)\,\mathrm dW_r.
\end{aligned}
\]
Taking expectations yields
\[
V(s,x)\ge \E\Big[\int_s^{t} f\big(r,X_r^\epsilon,\alpha_r^\epsilon\big)\,\mathrm dr
+ V\big(t,X_t^\epsilon\big)\Big]- C\epsilon(t-s).
\]
Since $\epsilon>0$ is arbitrary, we obtain the ``$\ge$'' direction of \eqref{eq:DPP}.
\end{proof}
\medskip

 Because the Hamiltonian may fail to attain its minimum under merely measurable
data, we use $\varepsilon$-selectors. Define
$$
H(t,x,P,a):=b(t,x,a)\cdot P+f(t,x,a),
\qquad (t,x,P,a)\in Q_T\times\mathbb R^d\times A.
$$
Since $A$ is a standard Borel space and $H$ is Borel measurable, the measurable
selection theorem (see, for example, \cite[Proposition~7.50]{BertsekasShreve1978}) yields that, for every $\epsilon>0$, there exists a
universally measurable map
$
\alpha^\epsilon:Q_T\times\mathbb R^d\to A
$
such that
\begin{equation}\label{eq:eps-selector}
H(t,x,P,\alpha^\epsilon(t,x,P))
\le
\inf_{a\in A}H(t,x,P,a)+\epsilon,
\qquad (t,x,P)\in Q_T\times\mathbb R^d .
\end{equation}
Once $P$ is fixed as a Borel measurable function of $(t,x)$, the feedback
selector
$
(t,x)\mapsto \alpha^\epsilon(t,x,P(t,x))
$
is universally measurable. Since $A$ is a standard Borel space, this feedback
selector can be modified on a Lebesgue null set of $Q_T$ to be Borel measurable
in $(t,x)$. In what follows, we use this Borel version, and the selection
inequality \eqref{eq:eps-selector} is understood to hold a.e. on $Q_T$.

Based on the above, we use $\epsilon$-selectors to construct near-optimal
controls, in the spirit of \cite{LimZhou1999}.

\begin{definition}[Near-optimality]
 A family of admissible weak control systems \(\{\mathbb A^{\epsilon}\} \subset
\mathfrak A_{s,T}(x)\) is called \emph{near-optimal} with respect to
\((s,x)\in Q_T\), if there is a modulus
\(\rho(\epsilon)\downarrow0\) such that
\(
|J(\mathbb A^{\epsilon};s,x)-V(s,x)|\le\rho(\epsilon)
\)
for every \(\epsilon>0\).  If \(\rho(\epsilon)=C\epsilon^{\delta}\)
with some \(\delta>0\), then the family is said to be near-optimal of order
$\varepsilon^\delta$.
\end{definition}

\smallskip

The verification theorem in the near-optimal case is presented as follows.

\begin{theorem}\label{thm:verification}
Let Assumption~\ref{ass:main} hold,
and let $u(\cdot,\cdot)\in W_p^{1,2}$ be the solution of the HJB equation \eqref{eq:hjb}. Then

{\rm(i)} \(u(s,x)\leq J(\mathbb A;s,x)\), for all \((s,x)\in Q_T\) and
\(\mathbb A\in\mathfrak A_{s,T}(x)\);

{\rm(ii)}  For any $\epsilon>0$, let $\alpha^\epsilon$ be the
$\varepsilon$-selector given by \eqref{eq:eps-selector},
and let $ X^{\epsilon}   $ solve
\begin{equation}
\md X_t^\epsilon
=
b\bigl(t,X_t^\epsilon,
\alpha^\epsilon(t,X_t^\epsilon,\nabla u(t,X_t^\epsilon))\bigr)\,\md t
+\sqrt2\,\md W_t,
\qquad X_s^\epsilon=x\label{eq:state-feed}.
\end{equation}
Then  the feedback control $\alpha^\epsilon_\cdot=\alpha^\epsilon(\cdot,X^\epsilon_{\cdot},\nabla u(\cdot,X^\epsilon_{\cdot}))$ induces an admissible weak control system and is near-optimal with respect to \((s,x)\in Q_T\) for Problem~\eqref{eq:oc_problem} with order $\epsilon$. More precisely,
$$
J(\alpha^\epsilon;s,x)\le V(s,x)+C\epsilon .
$$
\end{theorem}

\begin{proof}
The assertion (i) follows immediately from Theorem~\ref{Thm:exi-uni}.

For (ii),  by Lemma~\ref{lem:state-equation-solvability}{\rm(ii)},  SDE \eqref{eq:state-feed} admits a unique strong solution $ X^{\epsilon}   $ satisfying the Krylov estimate. Therefore $\alpha^\epsilon_\cdot:=\alpha^\epsilon(\cdot,X^\epsilon_{\cdot},\nabla u(\cdot,X^\epsilon_{\cdot}))$ induces an element of \(\mathfrak A_{s,T}(x)\).
Then applying the generalized It\^o formula to $u(t,X_t^\epsilon)$, and using \eqref{eq:eps-selector} together with the HJB equation, we obtain for $t\in[s,T]$,
\[
\begin{aligned}
u(s,x)
&=u(T,X_T^\epsilon) -\sqrt{2}\int_s^T \nabla u(r,X_r^\epsilon)\,\mathrm dW_r \\
&\quad -\int_s^T\big[(\partial_r u+\Delta u)(r,X_r^\epsilon)
+b(r,X_r^\epsilon,\alpha_r^\epsilon)\cdot\nabla u(r,X_r^\epsilon)\big]\,\mathrm dr \\
&\ge \int_s^T \big(f(r,X_r^\epsilon,\alpha_r^\epsilon)-\epsilon\big)\,\mathrm dr
-\sqrt{2}\int_s^T \nabla u(r,X_r^\epsilon)\,\mathrm dW_r.
\end{aligned}
\]
Taking expectations eliminates the martingale term, and yields
\[
u(s,x)\ge J(\alpha^\epsilon_\cdot;s,x)-\epsilon(T-s).
\]
Since $V(s,x)=u(s,x)$, this inequality together with (i) shows that $\alpha^\epsilon_\cdot$ is $\epsilon$-optimal with respect to \((s,x)\in Q_T\).
\end{proof}

\begin{cor}\label{cor:feedback-value}
Let Assumption~\ref{ass:main} hold.  For \((s,x)\in Q_T\), define the feedback
value function by
\[
V_{\rm fb}(s,x):=\inf_{\theta} J(\theta;s,x),
\]
where the infimum is taken over all Borel feedback selectors
\(\theta:Q_T\to A\), and \(J(\theta;s,x)\) denotes the cost of the corresponding
closed-loop system. Then
\[
V_{\rm fb}(s,x)=V(s,x),\qquad (s,x)\in Q_T .
\]
\end{cor}

\begin{proof}
Every Borel feedback selector induces an admissible weak control system by
Lemma~\ref{lem:state-equation-solvability}{\rm(ii)}, and hence
\[
V(s,x)\le V_{\rm fb}(s,x).
\]
Conversely, by Theorem~\ref{thm:verification}{\rm(i)}, if \(u\) is the Sobolev
solution of \eqref{eq:hjb}, then
\[
u(s,x)\le J(\theta;s,x)
\]
for every Borel feedback selector \(\theta\).  The
\(\varepsilon\)-selector in Theorem~\ref{thm:verification}{\rm(ii)} gives the
Borel feedback selector
\[
\theta^\varepsilon(t,y)
:=
\alpha^\varepsilon(t,y,\nabla u(t,y)),\quad (t,y)\in [s,T]\times\mathbb R^d,
\]
such that
\[
J(\theta^\varepsilon;s,x)\le V(s,x)+C\varepsilon
=u(s,x)+C\varepsilon .
\]
Letting \(\varepsilon\downarrow0\) yields
\[
V_{\rm fb}(s,x)\le u(s,x).
\]
Since \(u=V\) by Theorem~\ref{Thm:exi-uni}, the claim follows.
\end{proof}

\subsection{Proof of Theorem~\ref{Thm:exi-uni}: Policy Iteration}\label{subsec:PI}

We prove Theorem~\ref{Thm:exi-uni} by a policy iteration scheme tailored to the low-regularity control setting.

\begin{prop}\label{Prop: well}
Under Assumption~\ref{ass:main}, HJB equation \eqref{eq:hjb} admits a solution
$u\in W_{p}^{1,2}(Q_T)$.
\end{prop}

\begin{proof}
Fix $\delta>d/(2p)$ and put $w(x)=(1+|x|^2)^{-\delta}$, so that
$w\in L^p(Q_T)$. Write
\[
h(s,x,P)=\inf_{a\in A}H(s,x,P,a).
\]
Choose $u^0\in W_p^{1,2}(Q_T)$. Given $u^{k-1}$, apply the measurable
selection argument of \eqref{eq:eps-selector} to
$H(s,x,\nabla u^{k-1}(s,x),a)/w(x)$ with tolerance $2^{-k}$.
Taking a Borel version of the resulting selector gives a feedback
$\alpha^k:Q_T\to A$ satisfying, for a.e.\ $(s,x)\in Q_T$,
\begin{equation}\label{alpha-k}
 H(s,x,\nabla u^{k-1},\alpha^k)
 \le h(s,x,\nabla u^{k-1})+2^{-k}w(x).
\end{equation}

For $k\ge1$, let $u^{k}\in W_p^{1,2}(Q_T)$ solve the linear PDE:
\begin{equation}\label{eq:u-k}
\Bigg\{
\begin{aligned}
&\partial_s u^{k}+\Delta u^{k}+b(s,x,\alpha^{k}(s,x))\cdot\nabla u^{k}
+f(s,x,\alpha^{k}(s,x))=0,\ (s,x)\in Q_T,\\
& u^{k}(T,\cdot)=0.
\end{aligned}
\end{equation}
By $L^p$-parabolic regularity (e.g., \cite[Thm.~10.3]{KrylovRockner2005}, \cite[Thm.~5.1]{Zhang2011}),
\[
\|u^{k}\|_{W_p^{1,2}}\le C\|f(\cdot,\cdot,\alpha^k(\cdot,\cdot))\|_{L^p}\le C\|\Phi\|_{L^p}.
\]
For simplicity, we denote
\[
b(\alpha) = b(t,x,\alpha(t,x)),\quad f(\alpha) = f(t,x,\alpha(t,x)).
\]

\emph{Step 1: one-step comparison.}

From \eqref{eq:u-k} with $k-1$ and the definition of $\alpha^k$, for $k\ge2$,
\[
\begin{aligned}
0&=(\partial_s u^{k-1}+\Delta u^{k-1})+b(\alpha^{k-1})\!\cdot\!\nabla u^{k-1}+f(\alpha^{k-1})\\
&\ge(\partial_s u^{k-1}+\Delta u^{k-1})+b(\alpha^{k})\!\cdot\!\nabla u^{k-1}+f(\alpha^{k})-2^{-k}(1+|x|^2)^{-\delta}.
\end{aligned}
\]
Subtracting the equation for $u^{k}$ yields
\[
\begin{cases}
\partial_s(u^k-u^{k-1})+\Delta(u^k-u^{k-1})+b(\alpha^{k})\!\cdot\!\nabla(u^k-u^{k-1})
+2^{-k}(1+|x|^2)^{-\delta}\ge0,\\
(u^k-u^{k-1})(T,\cdot)=0.
\end{cases}
\]
Applying Lemma~\ref{lem:whole-space-max} with source $2^{-k}w$ gives
\begin{equation}\label{uk-uk-1}
u^{k}(s,x)-u^{k-1}(s,x)\le C\,2^{-k},\qquad (s,x)\in Q_T,
\end{equation}
where $C$ is independent of $k$ since $|b(\alpha^k)|\le\Phi$ and
$w\in L^p(Q_T)$. Hence $\{u^k+C2^{-k}\}_{k\ge1}$ is pointwise decreasing.
The uniform Sobolev bound and $p>d+2$ give a uniform lower bound, so there
exists a finite function $u^\infty$ with
\[
u^{k}(s,x) \to u^{\infty}(s,x)\quad\text{for all }(s,x)\in Q_T.
\]

The uniform $W_p^{1,2}$-bound implies $u^k\rightharpoonup u^\infty$ weakly in $W_p^{1,2}(Q_T)$, and $\|u^\infty\|_{W_p^{1,2}}\le C'\|\Phi\|_{L^p}$. Since $p>d+2$, Sobolev embedding yields a H\"older bound
\begin{equation}\label{ineq:u-k}
\sup_{k\ge1}\|u^{k}\|_{C_{t,x}^{(1+\alpha)/2,\,1+\alpha}}\le C\|\Phi\|_{L^p}
\end{equation}
for some $\alpha\in(0,1)$. Consequently, for any compact $\mathcal X\subset\R^d$,
\begin{equation}\label{eq:Vk-C-2}
\lim_{k\to\infty}\ \sup_{[0,T]\times\mathcal X}
\big(|u^k-u^\infty|+|\nabla u^k-\nabla u^\infty|\big)=0.
\end{equation}

\emph{Step 2: passage to the HJB equation.}
Since $|b(s,x,a)|\le\Phi(s,x)$ for all $a$, both $H(s,x,\cdot,a)$ and
$h(s,x,\cdot)$ are Lipschitz with constant $\Phi(s,x)$.
Set $P_k=\nabla u^k$ and $P_\infty=\nabla u^\infty$.
Applying these bounds and \eqref{alpha-k} at $P_{k-1}$ gives, a.e.\ on $Q_T$,
\[
0\le R_k:=H(s,x,P_k,\alpha^k)-h(s,x,P_k)
   \le 2\Phi\,|P_k-P_{k-1}|+2^{-k}w .
\]
By \eqref{ineq:u-k} and \eqref{eq:Vk-C-2}, the gradients are uniformly
bounded and converge pointwise to $P_\infty$. Since $\Phi,w\in L^p(Q_T)$,
dominated convergence yields $R_k\to0$ in $L^p(Q_T)$ and, using
\[
|h(s,x,P_k)-h(s,x,P_\infty)|\le\Phi\,|P_k-P_\infty|,
\]
also $h(\cdot,\cdot,P_k)\to h(\cdot,\cdot,P_\infty)$ in $L^p(Q_T)$.
Thus the lower-order term in \eqref{eq:u-k} converges strongly in $L^p$
to $h(\cdot,\cdot,\nabla u^\infty)$. Passing to the limit in
\eqref{eq:u-k} by weak $W_p^{1,2}$ convergence gives
\[
\partial_s u^\infty+\Delta u^\infty+h(s,x,\nabla u^\infty)=0
\quad\text{a.e.\ in }Q_T .
\]
Equation \eqref{eq:Vk-C-2} gives $u^\infty(T,x)=0$ for every $x$, as required.
\end{proof}

\medskip
Now we are in a position to conclude the proof of Theorem~\ref{Thm:exi-uni}.

\begin{proof}[Proof of Theorem~\ref{Thm:exi-uni}]
By Proposition~\ref{Prop: well}, it remains to prove (ii).
Uniqueness of the HJB solution then follows directly from (ii).

\emph{Step 1: $V\ge u$.}
Fix \(\mathbb A\in\mathfrak A_{s,T}(x)\), and let \((\alpha,X)\) be its
control-state pair.
By the generalized It\^o formula, for $t\in[s,T]$,
\[
\begin{aligned}
&u(t,X_t)
 =-\!\int_t^T\!\!\Big[(\partial_r u+\Delta u)(r,X_r)
+b(r,X_r,\alpha_r)\cdot\nabla u(r,X_r)\Big]\mathrm dr\\
&\qquad \qquad\ -\int_t^T\!\!\sqrt{2}\,\nabla u(r,X_r)\,\mathrm dW_r\\
 &\le \int_t^T f(r,X_r,\alpha_r)\,\mathrm dr
-\!\int_t^T\!\!\sqrt{2}\,\nabla u(r,X_r)\,\mathrm dW_r,
\end{aligned}
\]
where we used that $u$ solves the HJB equation \eqref{eq:hjb}. Taking expectations at $t=s$ and using that the stochastic integral is a martingale with zero mean, we obtain
\[
u(s,x)\le \mathbb E^{\mathbb A}\!\int_s^T f(r,X_r,\alpha_r)\,\mathrm dr
=J(\mathbb A;s,x).
\]
Taking the infimum over \(\mathbb A\in\mathfrak A_{s,T}(x)\) yields \(u\le V\).

 {\it Step 2:} $V\le u$.
Fix $\epsilon>0$, and let $\alpha^\epsilon$ be the $\epsilon$-selector given by
\eqref{eq:eps-selector}. Set
$$
\bar\alpha^\epsilon(t,x):=
\alpha^\epsilon(t,x,\nabla u(t,x)),
\qquad (t,x)\in Q_T .
$$
Taking $P=\nabla u(t,x)$ in \eqref{eq:eps-selector}, we obtain
\begin{equation}\label{inequ: bf-leq}
\begin{aligned}
& b(t,x,\bar\alpha^\epsilon(t,x))\cdot\nabla u(t,x)+f(t,x,\bar\alpha^\epsilon(t,x))\\
&\le \inf_{a\in A}\!\Big\{b(t,x,a)\cdot\nabla u(t,x)+f(t,x,a)\Big\}+\varepsilon.
\end{aligned}
\end{equation}

By Lemma~\ref{lem:state-equation-solvability}{\rm(ii)}, the closed-loop SDE
\begin{equation}\label{eq:state-feedback}
\mathrm d\bar  X_t^\epsilon
=b(t,\bar  X_t^\epsilon,\bar \alpha^\epsilon(t,\bar X_t^\epsilon))\,\mathrm dt+\sqrt{2}\,\mathrm dW_t,
\qquad \bar  X_s^\epsilon=x\in\R^d .
\end{equation}
admits a unique strong solution. Define the feedback control
\(\bar\alpha_t^\epsilon:=\bar\alpha^\epsilon(t,\bar X_t^\epsilon)\).
Applying the generalized It\^o formula to $u(t,\bar X_t^{\epsilon})$ and using \eqref{inequ: bf-leq} together with the HJB equation, we obtain, for $t\in[s,T]$,
\[
\begin{aligned}
&u(t,\bar X_t^{\epsilon})
 =-\!\int_t^T\!\!\Big[(\partial_r u+\Delta u)(r,\bar X_r^{\epsilon})
+b(r,\bar X_r^{\epsilon},\bar \alpha_r^{\epsilon})\cdot\nabla u(r,\bar X_r^{\epsilon})\Big]\mathrm dr \\
&\qquad\qquad \ -\int_t^T\!\!\sqrt{2}\,\nabla u(r,\bar X_r^{\epsilon})\,\mathrm dW_r\\
&\ge -\!\int_t^T\!\!\Big[(\partial_r u+\Delta u)(r,\bar X_r^{\epsilon})
+\inf_{a\in A}\{b(r,\bar X_r^{\epsilon},a)\cdot\nabla u(r,\bar X_r^{\epsilon})+f(r,\bar X_r^{\epsilon},a)\}\\
&\hspace{6.2em}\qquad -\,f(r,\bar X_r^{\epsilon},\bar \alpha_r^{\epsilon})+C\epsilon\Big]\mathrm dr
-\!\int_t^T\!\!\sqrt{2}\,\nabla u(r,\bar X_r^{\epsilon})\,\mathrm dW_r\\
&=\int_t^T\!\!\Big[f(r,\bar X_r^{\varepsilon},\bar \alpha_r^{\epsilon})-C\epsilon\Big]\mathrm dr
-\!\int_t^T\!\!\sqrt{2}\,\nabla u(r,\bar X_r^{\epsilon})\,\mathrm dW_r,
\end{aligned}
\]
where the last equality uses that $u$ solves the HJB equation.
Taking expectations at $t=s$ and using again that the stochastic integral has zero mean, we get, for $(s,x)\in Q_T$,
\[
\begin{aligned}
u(s,x)&\ge \E\!\int_s^T f(r,\bar X_r^{\epsilon},\bar \alpha_r^{\epsilon})\,\mathrm dr - C\varepsilon(T-s)\\
&= J(\bar \alpha^{\epsilon}_{\cdot};s,x)-C\epsilon (T-s) \ge V(s,x)-C\epsilon (T-s) .
\end{aligned}\]
Since $\epsilon>0$ is arbitrary, we conclude $V\le u$.

Combining the two steps gives $V=u$ and completes the proof of (ii).
\end{proof}

\smallskip
 At this stage, we have accomplished the first objective of this study.
Starting from the HJB equation \eqref{eq:hjb}, we have established key
connections in control theory, namely the relationship between the Sobolev
solution of the HJB equation \eqref{eq:hjb} and the value function of
Problem~\eqref{eq:oc_problem}, the dynamic programming principle, and the
stochastic verification theorem under these low regularity assumptions.

\section{Mollification Scheme and Convergence Analysis}\label{sec:mollify}

A classical approach to SDEs with irregular drift is to smooth the drift coefficient, solve the regularized equation, and then pass to the limit as the smoothing parameter \(\varepsilon\to 0\).
In our optimal control setting, however, this
approach  behaves in a subtler way under Assumption~\ref{ass:main}.
Now we begin
by smoothing the data \(b\) and \(f\).

\smallskip
\paragraph{Mollifier}
Fix a non-negative kernel
\(\zeta\in C_{c}^{\infty}(\R\times\R^{d})\) supported in the unit ball and
normalized by $ \iint_{\mathbb R\times\mathbb R^d}\zeta(t,x)\,\md x\,\md t=1.$
For \(\varepsilon>0\), set
\[
  \zeta_{\varepsilon}(t,x)
  :=\varepsilon^{-(d+1)}
     \,\zeta(t/\varepsilon,x/\varepsilon),\quad (t,x)\in\mathbb R\times\mathbb R^d.
\]
Extend \(g=b,f\) and $\Phi$ by zero for \(t\notin[0,T]\) and define the convolutions
\[
  g_{\varepsilon}(t,x,a)
  :=(\zeta_{\varepsilon}*g(\cdot,\cdot,a))(t,x),\quad (t,x)\in Q_T,\
    a\in A .
\]
Then the standard convolution properties and Assumption~\ref{ass:main} show that, for all $(t,x)\in Q_T$ and $a\in A$,
\begin{equation}\label{eq:b-f-eps}
\begin{aligned}
& |b_\varepsilon(t,x,a)|+|f_\varepsilon(t,x,a)|
 \le \Phi_\varepsilon(t,x):=(\zeta_\varepsilon*\Phi)(t,x),\\
& b_\varepsilon(\cdot,\cdot,a),\ f_\varepsilon(\cdot,\cdot,a)\in C^\infty(Q_T),\\
&b_\varepsilon(\cdot,\cdot,a),\ f_\varepsilon(\cdot,\cdot,a)
 \text{ are Lipschitz in }(t,x)
 \text{ for each fixed }\varepsilon>0,\\
& b_\varepsilon(\cdot,\cdot,a)\to b(\cdot,\cdot,a),\quad
  f_\varepsilon(\cdot,\cdot,a)\to f(\cdot,\cdot,a)
 \quad\text{in }L^p(Q_T),\quad \varepsilon\downarrow0.
\end{aligned}
\end{equation}
 Moreover, $\|\Phi_\varepsilon\|_{L^p}
 \le \|\Phi\|_{L^p}$.

\smallskip
\paragraph{Regularized control problems}
For \(\varepsilon>0\) and \((s,x)\in Q_T\), let
\(\mathfrak A^\varepsilon_{s,T}(x)\) be the class of admissible weak control
systems defined as above, with \(b\) replaced by \(b_\varepsilon\). For
\(\mathbb A^\varepsilon\in\mathfrak A^\varepsilon_{s,T}(x)\), define
\begin{align}
 & J_{\varepsilon}(\mathbb A^\varepsilon;s,x) :=
     \mathbb E^{\mathbb A^\varepsilon}\int_{s}^{T}
     f_{\varepsilon}(t,X_t,\alpha_t)\md t,
     \label{eq:cost-eps}\\
 & V_{\varepsilon}(s,x) :=
     \inf_{\mathbb A^\varepsilon\in\mathfrak A^\varepsilon_{s,T}(x)}
     J_{\varepsilon}(\mathbb A^\varepsilon;s,x). \label{eq:value-eps}
\end{align}
Since $b_{\varepsilon}$ and $f_{\varepsilon}$ are bounded and Lipschitz in
$(t,x)$ for each fixed $\varepsilon>0$, the above regularized control problem is
well defined.
Consider the following regularized HJB equation:
\begin{equation}\label{eq:HJB-eps}
\left\{\!\!
\begin{aligned}
&(\partial_s u_\varepsilon+\Delta u_\varepsilon)(s,x)
+\inf_{a\in A}
\{b_\varepsilon(s,x,a)\cdot\nabla u_\varepsilon(s,x)
+f_\varepsilon(s,x,a)\}=0,
\ (s,x)\in Q_T,\\
&u_\varepsilon(T,x)=0,\qquad x\in\mathbb R^d.
\end{aligned}
\right.
\end{equation}
Since the regularized coefficients satisfy the assumptions of
Theorem~\ref{thm:verification}, the same verification argument yields that
$V_\varepsilon$ coincides with the unique bounded Sobolev solution
$u_\varepsilon$ of \eqref{eq:HJB-eps}.

\subsection{Main results}\label{subsec:moll-main}

It is natural to expect that, as $\varepsilon\downarrow0$, the regularized
value functions $V_\varepsilon$ converge to the original value function $V$,
just as solutions of SDEs with mollified coefficients converge to those of the
original equation. Our analysis shows, however, that such convergence is \emph{not}
unconditional.
Under Assumption~\ref{ass:main}, the following one-sided estimate holds.
\begin{prop}[Liminf inequality]\label{thm:liminf}
For any compact $\mathcal X\subset\mathbb R^d$,
$$
 V(s,x)\;\le\;\liminf_{\varepsilon\to0} V_\varepsilon(s,x)
 \quad \text{for all } (s,x)\in[0,T]\times\mathcal X.
$$
\end{prop}

The reverse inequality may \emph{fail} in general---see Example~\ref{ex:failure-Lp} for a counterexample.
To recover full convergence $V_\varepsilon\to V$, we identify a natural structural restriction on the control space.
Specifically, if the action set $A$ is countable, the approximation procedure becomes stable enough to ensure the missing inequality.

\begin{prop}\label{Prop:V>V_n}
If the control set $A$ is countable,
then
$$
 V(s,x)\;\ge\;\limsup_{\varepsilon\to0} V_\varepsilon(s,x)
 \quad \text{for all } (s,x)\in Q_T.
$$
\end{prop}

Propositions~\ref{thm:liminf} and~\ref{Prop:V>V_n} give pointwise convergence.
The uniform H\"older bound at the start of Section~\ref{subsec:proof-liminf}
makes $\{V_\varepsilon\}$ equicontinuous on each compact set, so convergence
is locally uniform.

\begin{theorem}[Convergence under countable actions]\label{thm:countable}
Let Assumption~\ref{ass:main} hold and suppose the control set $A$ is countable.
Then $V_\varepsilon\to V$ locally uniformly on $Q_T$.
\end{theorem}

\smallskip
The proofs of Propositions~\ref{thm:liminf} and~\ref{Prop:V>V_n} will be given in the subsequent subsections, following the counterexample below.

\medskip

\begin{example}[Failure of convergence under mollification]
\label{ex:failure-Lp}
Let $d=1$ and let $A=\mathbb R$. Fix $s_0\in(0,T)$ and choose
$\delta>0$ such that $s_0+\delta<T$. Let
$\chi(t)={\bf 1}_{[s_0,s_0+\delta]}$.
For $\lambda\ge0$, set
$$
q_\lambda(x):=e^{-\lambda x^2},
\qquad
F(x):=e^{-x^2}\big(2+\tanh x\big).
$$
We first take $\lambda>0$. Then $q_\lambda,F\in L^p(\mathbb R)$ for every
$p\ge1$, and $F\ge0$. In the control problem \eqref{eq:oc_problem}, define
$$
b_\lambda(t,x,a):=\chi(t)q_\lambda(x)\mathbf 1_{\{x\ne a\}},
\qquad
f(t,x,a):=\chi(t)F(x).
$$
These coefficients satisfy Assumption~\ref{ass:main} with
$
\Phi_\lambda(t,x):=\chi(t)\big(q_\lambda(x)+F(x)\big).
$

For notational simplicity, we suppress the dependence on $\lambda$ in the value
functions. Consider the initial point $(s_0,0)$. In the original control
problem, choose the feedback control $\alpha_t=X_t$. Then the drift vanishes
and the corresponding state is
$
X_t^0=\sqrt2(W_t-W_{s_0})
$, $t\in[s_0,T]$. Hence
\begin{equation}\label{eq:V-upper-Lp-example}
V(s_0,0)
\le
\mathbb E\int_{s_0}^{s_0+\delta}F(X_t^0)\,\md t .
\end{equation}

On the other hand, since the set $\{x=a\}$ is Lebesgue null, mollification in
the time-space variables gives, for every $a\in A$,
$$
b_{\lambda,\varepsilon}(t,x,a)
=
\zeta_\varepsilon*(\chi q_\lambda)(t,x)
=: \bar b_{\lambda,\varepsilon}(t,x),
\qquad
f_\varepsilon(t,x,a)
=
\zeta_\varepsilon*(\chi F)(t,x)
=: \bar f_\varepsilon(t,x).
$$
Thus the regularized control problem has no effective control in the
coefficients and is given by
$$
\md X_t^\varepsilon
=
\bar b_{\lambda,\varepsilon}(t,X_t^\varepsilon)\,\md t
+\sqrt2\,\md W_t,
\qquad
X_{s_0}^\varepsilon=0,
$$
with value
$
V_\varepsilon(s_0,0)
=
\mathbb E\int_{s_0}^T
\bar f_\varepsilon(t,X_t^\varepsilon)\,\md t .
$

Since
$f_\varepsilon\to\chi F$ and
$b_{\lambda,\varepsilon}\to\chi q_\lambda$ as $\varepsilon\to0$, the standard
stability of SDEs gives
\begin{equation}\label{eq:Ve-upper-Lp-example}
V_\varepsilon(s_0,0)
\longrightarrow
\mathbb E\int_{s_0}^{s_0+\delta} F(Y_t)\,\md t,
\end{equation}
where
$
\md Y_r=q_\lambda(Y_r)\,\md r+\sqrt2\,\md W_r,
$ $
Y_{s_0}=0,
$ $r\in[s_0,s_0+\delta]$.

It remains to compare \eqref{eq:Ve-upper-Lp-example} and
\eqref{eq:V-upper-Lp-example}. Since
$
F'(0)=1
$
and
$
q_\lambda(0)=1,
$
the generator expansion at $x=0$ yields, as $r\downarrow s_0$,
$$
\mathbb E[F(Y_r)]
-
\mathbb E[F(X_r^0)]
=
(r-s_0)q_\lambda(0)F'(0)+o(r-s_0)
=
(r-s_0)+o(r-s_0).
$$
Thus, by taking $\delta>0$ sufficiently small, we obtain
$$
\mathbb E\int_{s_0}^{s_0+\delta} F(Y_r)\,\md r
>
\mathbb E\int_{s_0}^{s_0+\delta} F(X_r^0)\,\md r
\ge
V(s_0,0).
$$
Therefore,
$
\lim_{\varepsilon\to0}V_\varepsilon(s_0,0)
>
V(s_0,0).
$
This shows that the convergence $V_\varepsilon\to V$ may fail even under the
global $L^p$-integrability condition imposed in Assumption~\ref{ass:main}.

\smallskip

The limiting case $\lambda=0$ gives $q_0\equiv1$. It is outside the global
$L^p$ framework on $Q_T$, but the coefficients
$
b_0(t,x,a)=\chi(t)\mathbf 1_{\{x\ne a\}}
$
and $f(t,x,a)=\chi(t)F(x)$ are bounded Borel measurable. The same calculation
applies, because $q_0(0)=1$. Hence the failure of convergence is not
caused by unbounded coefficients; it may also occur in a bounded-coefficient
setting outside Assumption~\ref{ass:main}.
\end{example}

\subsection{Proof of Proposition~\ref{thm:liminf}}\label{subsec:proof-liminf}

We prove the estimate
$V(\cdot,\cdot) \leq \liminf_{\varepsilon\to0}V_\varepsilon(\cdot,\cdot)$.
By \eqref{eq:b-f-eps} and
$\|\Phi_\varepsilon\|_{L^p}\le \|\Phi\|_{L^p}$, the Sobolev estimates for the
regularized HJB equation \eqref{eq:HJB-eps} imply that
$\sup_{\varepsilon>0}\|u_{\varepsilon}\|_{W^{1,2}_{p}}<\infty$.
Since $p>d+2$, the Sobolev embedding theorem gives, for some
$\alpha\in(0,1)$,
$$
\sup_{\varepsilon>0}\|u_{\varepsilon}\|_{C_{t,x}^{(1+\alpha)/2,1+\alpha}}\le C\sup_{\varepsilon>0}\|u_{\varepsilon}\|_{W^{1,2}_{p}}<\infty.
$$
From any sequence tending to zero we may extract a subsequence $\varepsilon_k$
such that $u_{\varepsilon_k}\rightharpoonup\tilde u$ in $W_p^{1,2}(Q_T)$.
By Arzel\`a--Ascoli, after passing to a further subsequence, for any compact
$\mathcal X\subset\mathbb R^d$ we have
\begin{equation} \label{u-tilde u}
\lim_{k\to \infty}\sup_{(s,x)\in [0,T]\times\mathcal X}\big(|u_{\varepsilon_k}(s,x)-\tilde{u}(s,x)|+|\nabla u_{\varepsilon_k}(s,x)-\nabla\tilde{u}(s,x)|\big)=0.
\end{equation}

Since $u_{\varepsilon_k}(T,\cdot)=0$, this convergence gives
$\tilde u(T,\cdot)=0$. We next characterize $\tilde u$.

\smallskip

\begin{lem}\label{weak-super} Under Assumption \ref{ass:main},
for any  nonnegative  test function $\phi(\cdot,\cdot)\in C_c^{1,\infty}({Q_T};\mathbb R)$, one has
$$\begin{array}{ll}
&\displaystyle
\iint_{Q_T}-\big(\tilde{u}(s,x)\partial_{s}\phi(s,x)+\nabla\tilde{u}(s,x)\cdot\nabla\phi(s,x)\big)\md x\md s\\
&\displaystyle\qquad +\iint_{Q_T}\inf_{a\in A}\big\{ b(s,x,a)\cdot\nabla\tilde{u}(s,x)+f(s,x,a)\big\}\phi(s,x)\md x\md s \le0.
\end{array}$$

\end{lem}

 Here, a function is called {\emph{a weak supersolution}} of the HJB
equation~\eqref{eq:hjb} if it satisfies the above integral inequality for
every nonnegative test function
$\phi\in C_c^{1,\infty}(Q_T;\mathbb R)$.

\smallskip
\begin{proof}
 For each $k\ge 1$, since $V_{\varepsilon_k}=u_{\varepsilon_k}$ and
$u_{\varepsilon_k}$ is the bounded Sobolev solution of the regularized HJB
equation \eqref{eq:HJB-eps}, it satisfies the weak formulation.
Indeed, multiplying \eqref{eq:HJB-eps} by any nonnegative
$\phi\in C_c^{1,\infty}(Q_T;\mathbb R)$ and integrating by parts over $Q_T$, we obtain
\begin{equation}\label{Weak-Sou}\begin{array}{ll}
&\displaystyle \iint_{Q_T}-\big(u_{\varepsilon_k}(s,x)\partial_{s}\phi(s,x)+\nabla u_{\varepsilon_k}(s,x)\cdot\nabla\phi(s,x)\big)\md x\md s \\
&\displaystyle\quad   +\iint_{Q_T}\inf_{a\in A}\big\{ b_{\varepsilon_k}( s,x,a)\cdot\nabla u_{\varepsilon_k}(s,x)+f_{\varepsilon_k}(s,x,a)\big\}\phi(s,x)\md x\md s =0.
\end{array}\end{equation}

For convenience, for $(s,x,P)\in Q_T\times \mathbb R^d$, we set
$$ \begin{aligned}
&G(s,x,P):=
\inf_{a\in A}
\left\{
b(s,x,a)\cdot P+f(s,x,a)
\right\},\\
&G_k (s,x,P):=
\inf_{a\in A}
\left\{
b_{\varepsilon _k}(s,x,a)\cdot P+f_{\varepsilon _k}(s,x,a)
\right\}.
\end{aligned}$$
Since   $\nabla\widetilde u$ is locally bounded,
we have $G(\cdot,\cdot,\nabla\tilde{u} (\cdot,\cdot)) \in L^p_{\rm loc}(Q_T)$. Hence
$$
\zeta_{\varepsilon_k}*
G(\cdot,\cdot,\nabla\widetilde u(\cdot,\cdot))
\to
G(\cdot,\cdot,\nabla\widetilde u(\cdot,\cdot))
\quad\text{in }L^p_{\rm loc}(Q_T).
$$
Therefore, for  $\phi\in C_c^{1,\infty}(Q_T;\mathbb R)$,
\begin{equation}\label{Gm-G}
\iint_{Q_T}
 [\zeta_{\varepsilon_k}*
G(\cdot,\cdot,\nabla\widetilde u(\cdot,\cdot))
](s,x) \phi(s,x)\md x\md s
\to
\iint_{Q_T}
G(s,x,\nabla\tilde{u} (s,x))\phi(s,x)\md x \md s .
\end{equation}
Moreover, using  \eqref{u-tilde u}, for   any  $\widetilde\eps>0$, there is some $K$ such that for any  $k>K$,  and for all $(s,x)\in\operatorname{supp}\phi$,
$$\begin{array}{ll}
 & \displaystyle \big|G_k(s,x,\nabla{u}_{\varepsilon_k}(s,x))-G_k(s,x,\nabla\tilde{u} (s,x)) \big|\phi(s,x)\\
 & \displaystyle \le\sup_{a\in A}\big|b_{\varepsilon_k}(s,x,a)\cdot\nabla u_{\varepsilon_k}(s,x)-b_{\varepsilon_k}(s,x,a)\cdot\nabla\tilde{u}(s,x)\big|\phi(s,x) \\
 & \displaystyle \le \Phi_{\varepsilon_k }(s,x)\phi(s,x)\widetilde\varepsilon .
\end{array}$$
Then, using \eqref{Weak-Sou}, for all $k>K$,
$$\begin{array}{ll}
 &\displaystyle \iint_{Q_T}-\big(u_{\varepsilon_k}(s,x)\partial_{s}\phi(s,x)+\nabla u_{\varepsilon_k}(s,x)\cdot\nabla\phi(s,x)\big)\md x\md s\\
 &\displaystyle  \qquad +\iint_{Q_T}G_k(s,x,\nabla\tilde{u} (s,x))\phi(s,x)\md x\md s \\
 & \displaystyle \le\iint_{Q_T}-\big(u_{\varepsilon_k}(s,x)\partial_{s}\phi (s,x)+\nabla u_{\varepsilon_k}(s,x)\cdot\nabla\phi(s,x)\big)\md x\md s\\
 &\displaystyle \qquad  +\iint_{Q_T}\!\!\!\Big(G_k(s,x,\nabla{u}_{\varepsilon_k}(s,x))+\widetilde\varepsilon  \Phi_{\varepsilon_k}(s,x) \Big)\phi(s,x)\md x\md s \\
 &\displaystyle  =\widetilde\varepsilon \iint_{Q_T} \Phi_{\varepsilon_k}(s,x)\phi(s,x)\md x\md s .
\end{array}
$$

Taking $\liminf_{k\to\infty}$ on both sides,  we get
$$
\begin{array}{ll}
 &\displaystyle
 \widetilde\varepsilon \iint_{Q_T}\Phi(s,x)\phi(s,x)\md x\md s
 +\iint_{Q_T}\big(\tilde{u}(s,x)\partial_{s}\phi(s,x)
 +\nabla\tilde{u}(s,x)\cdot\nabla \phi(s,x)\big)\md x\md s
\\[2mm]
 & \displaystyle \geq
 \liminf_{k\to\infty}\iint_{Q_T}
 G_k(s,x,\nabla\tilde{u}(s,x))\phi(s,x)\md x\md s
\\[2mm]
 & \displaystyle =
 \liminf_{k\to\infty}\iint_{Q_T}\inf_{a\in A}
 \big\{
 [\zeta_{\varepsilon_k}*b(\cdot,\cdot,a)](s,x)
 \cdot\nabla\tilde{u}(s,x)
 \\
 &\hskip 4cm+[\zeta_{\varepsilon_k}*f(\cdot,\cdot,a)](s,x)
 \big\}\phi(s,x)\md x\md s
\\[2mm]
 &\displaystyle  \geq {\rm(I)}+{\rm(II)}.
\end{array}
$$
where
$$
\begin{array}{ll}
&\displaystyle   {\rm(I)} :=
\liminf_{k\to\infty}\iint_{Q_T}\inf_{a\in A}
\big\{
[\zeta_{\varepsilon_k}*
\big(b(\cdot,\cdot,a)\cdot\nabla\tilde{u}(\cdot,\cdot)
+f(\cdot,\cdot,a)\big)](s,x)
\big\}\phi(s,x)\md x\md s
\\
[3mm] &\displaystyle \quad \ \ \ge
\liminf_{k\to\infty}\iint_{Q_T}
[\zeta_{\varepsilon_k}*
G(\cdot,\cdot,\nabla\tilde{u}(\cdot,\cdot))](s,x)
\,\phi(s,x)\md x\md s
\\
[3mm] &\displaystyle\quad \ \  =
\iint_{Q_T}G(s,x,\nabla\tilde{u}(s,x))\phi(s,x)\md x\md s,
\qquad \mbox{using }\eqref{Gm-G}.
\end{array}
$$
and
$$
\begin{array}{ll}
&\displaystyle {\rm(II)}
  :=
\liminf_{k\to\infty}\iint_{Q_T}\inf_{a\in A}
\big\{
[\zeta_{\varepsilon_k}*b(\cdot,\cdot,a)](s,x)
\cdot\nabla\tilde{u}(s,x)
\\
&\hskip 5cm
-
[\zeta_{\varepsilon_k}*
\big(b(\cdot,\cdot,a)\cdot\nabla\tilde{u}(\cdot,\cdot)\big)](s,x)
\big\}\phi(s,x)\md x\md s
\\[2mm]
&\displaystyle\quad\ \  \ge
-\limsup_{k\to\infty}\iint_{Q_T}\sup_{a\in A}
\big|
[\zeta_{\varepsilon_k}*b(\cdot,\cdot,a)](s,x)
\cdot\nabla\tilde{u}(s,x)
\\
&\hskip 5cm
-
[\zeta_{\varepsilon_k}*
\big(b(\cdot,\cdot,a)\cdot\nabla\tilde{u}(\cdot,\cdot)\big)](s,x)
\big|\phi(s,x)\md x\md s .
\end{array}
$$

Since $\phi$ has compact support, we choose a compact set $\mathcal X\subset Q_T$
containing a small neighbourhood of $\operatorname{supp}\phi$. For $k$ large
enough, if $(s,x)\in\operatorname{supp}\phi$ and
$\zeta_{\varepsilon_k}(s-r,x-y)\neq0$, then $(r,y)\in \mathcal X$. Hence, by the
bound \(|b|\le\Phi\) and the H\"older continuity of \(\nabla\widetilde u\) on
$\mathcal X$, we obtain
$$
\begin{aligned}
&\sup_{a\in A}\Big|
[\zeta_{\varepsilon_k}*b(\cdot,\cdot,a)](s,x)\cdot\nabla\tilde u(s,x)
-
[\zeta_{\varepsilon_k}*(b(\cdot,\cdot,a)\cdot\nabla\tilde u(\cdot,\cdot))](s,x)
\Big|  \\
&=
\sup_{a\in A}\left|
\iint_{Q_T}
\zeta_{\varepsilon_k}(s-r,x-y)b(r,y,a)\cdot
\big(\nabla\tilde u(s,x)-\nabla\tilde u(r,y)\big)\md y\md r
\right|  \\
&\le
 C_\mathcal X
\iint_{Q_T}
\zeta_{\varepsilon_k}(s-r,x-y)\Phi(r,y)
\big(|s-r|^{\alpha/2}+|x-y|^\alpha\big)\md y\md r .
\end{aligned}$$
 where
$$
 C_\mathcal X:=
[\nabla\widetilde u\, ]_{\frac\alpha2,\alpha;\mathcal X}
:=
\sup_{\substack{(s,x),(r,y)\in \mathcal X\\ (s,x)\ne(r,y)}}
\frac{
|\nabla\widetilde u(s,x)-\nabla\widetilde u(r,y)|
}{
|s-r|^{\alpha/2}+|x-y|^\alpha
}<\infty .
$$
Then
$$\begin{array}{ll}
& {\rm(II)}  \ge -C_\mathcal X \limsup_{k\to\infty} \iint_{Q_T}\phi(s,x)\Big(
\iint_{Q_T}
\zeta_{\varepsilon_k}(s-r,x-y)\Phi(r,y)  \\
&\hskip 6.5cm \cdot
\big(|s-r|^{\frac \alpha 2}+|x-y|^\alpha\big)
 \md y\md r\Big)\md x\md s \\
[2mm]&\quad \ \ \ge
-C_\mathcal X\|\phi\|_\infty
\limsup_{k\to\infty}
\iint_{\mathcal X}\Phi(r,y)\,\md y \md r
\iint_{{\mathbb R\times\mathbb R^d}}
\zeta_{\varepsilon_k}(\tau,z)
\big(|\tau|^{\alpha/2}+|z|^\alpha\big)
\,\md z\md \tau
\\
[2mm]&\quad \ \ \ge
-C\,C_\mathcal X\|\phi\|_\infty\|\Phi\|_{L^p(\mathcal X)}
\limsup_{k\to\infty}
\iint_{{\mathbb R\times\mathbb R^d}}
\zeta_{\varepsilon_k}(\tau,z)
\big(|\tau|^{\alpha/2}+|z|^\alpha\big)
\,\md z \md \tau  \\
[3mm]&\quad \ \ \ge
-C\,C_\mathcal X\|\phi\|_\infty\|\Phi\|_{L^p(\mathcal X)}
\limsup_{k\to\infty}
\big(\varepsilon_k^{\alpha/2}+\varepsilon_k^\alpha\big) =0
.\end{array}
$$
To sum up, one has
$$\begin{array}{ll}
 & \displaystyle    \widetilde\varepsilon \iint_{Q_T}\!\Phi(s,x)\phi(s,x)\md x\md s +\iint_{Q_T}\!\big(\tilde{u}(s,x)\partial_{s}\phi(s,x)+\nabla\tilde{u}(s,x)\cdot\nabla \phi(s,x)\big)\md x\md s \\
 & \displaystyle \geq\iint_{Q_T}G(s,x, \nabla\tilde{u}(s,x))\phi(s,x) \md x\md s.
\end{array}$$
Due to the arbitrariness of $\widetilde\varepsilon $, we complete the proof.
\end{proof}

\smallskip

We now prove Proposition~\ref{thm:liminf}.

\begin{proof}[Proof of Proposition~\ref{thm:liminf}]
By Theorem~\ref{Thm:exi-uni}, \(V\in W_p^{1,2}(Q_T)\) solves the HJB equation
\eqref{eq:hjb}; in particular, it is a weak subsolution of \eqref{eq:hjb}.
According to Lemma~\ref{weak-super} and the comparison principle
in Lemma~\ref{lem:HJB-weak-comparison}, we have
$V(s,x)\leq \tilde{u}(s,x)$ for all $(s,x)\in Q_T$.

For fixed $(s,x)$, choose the initial sequence to realize
$\liminf_{\varepsilon\downarrow0}u_\varepsilon(s,x)$. Its subsequence
in \eqref{u-tilde u} has the same limit at $(s,x)$, so
$V(s,x)\le\tilde u(s,x)=\liminf_{\varepsilon\downarrow0}V_\varepsilon(s,x)$.
The proof is complete.
\end{proof}

\subsection{Proof of Proposition~\ref{Prop:V>V_n}}\label{subsec:proof-countable}

List the elements of $A$ without repetitions as $a_1,a_2,\ldots$.
If $A$ is finite, the sums below stop at $|A|$ and the remaining tails are empty.

\begin{lemma}\label{lem:composition-convergence}
Assume that $A$ is countable, and let
$\theta:Q_T\to A$ be a fixed Borel measurable feedback selector. Let
$X^{\varepsilon,\theta}$ and $X^\theta$ denote the unique strong solutions of
the following SDEs:
$$\begin{aligned}
&\md X_t^{\varepsilon,\theta}
=
b_\varepsilon\bigl(t,X_t^{\varepsilon,\theta},
\theta(t,X_t^{\varepsilon,\theta})\bigr)\,\md t
+\sqrt2\,\md W_t,\qquad
X_s^{\varepsilon,\theta}=x,\quad t\in[s,T],\\
&\md X_t^\theta
=
b\bigl(t,X_t^\theta,\theta(t,X_t^\theta)\bigr)\,\md t
+\sqrt2\,\md W_t,\qquad
X_s^\theta=x,\quad t\in[s,T].
\end{aligned}$$
The corresponding feedback costs are
$$
\begin{aligned}
&J_\varepsilon(\theta;s,x)
 :=
\mathbb E\int_s^T
f_\varepsilon(t,X_t^{\varepsilon,\theta},
\theta(t,X_t^{\varepsilon,\theta}))\,\md t,\\
&J(\theta;s,x)
 :=
\mathbb E\int_s^T
f(t,X_t^\theta,\theta(t,X_t^\theta))\,\md t .
\end{aligned}
$$
Then, as  $ \varepsilon\downarrow0 $,

{\rm(i)}
$
g_\varepsilon(\cdot,\cdot,\theta(\cdot,\cdot))
\longrightarrow
g(\cdot,\cdot,\theta(\cdot,\cdot))
$  in $L^p(Q_T),
$ for $g=b,f$;

{\rm(ii)}
$
J_\varepsilon(\theta;s,x)\longrightarrow J(\theta;s,x),
$   for all $(s,x)\in Q_T$.
\end{lemma}

\begin{proof}
Since $\theta$ is Borel measurable, the composed drifts
$b_\varepsilon(t,x,\theta(t,x))$ and $b(t,x,\theta(t,x))$ are deterministic
Borel functions. Moreover,
$
|b_\varepsilon(t,x,\theta(t,x))|\le\Phi_\varepsilon(t,x)
$
and
$
|b(t,x,\theta(t,x))|\le\Phi(t,x).
$
By Lemma~\ref{lem:state-equation-solvability}{\rm(ii)}, both state equations
therefore admit unique strong solutions and satisfy Krylov estimates. The
constants in the estimates are uniform in $\varepsilon$, because
$
\|\Phi_\varepsilon\|_{L^p}\le\|\Phi\|_{L^p}.
$

For $i\ge1$, set
$$
E_i:=\{(t,x)\in Q_T:\theta(t,x)=a_i\}.
$$
Since
$A=\{a_1,a_2,\ldots\}$ is countable and each $E_i$ is Borel measurable,
the family $\{E_i\}_{i\ge1}$ is a measurable partition of $Q_T$. Hence, for $g=b,f$,
we have
$$
g_\varepsilon(t,x,\theta(t,x))-g(t,x,\theta(t,x))
=
\sum_{i=1}^\infty
\bigl(g_\varepsilon(t,x,a_i)-g(t,x,a_i)\bigr)\mathbf 1_{E_i}(t,x).
$$
By \eqref{eq:b-f-eps},
  for every fixed $N\ge1$,
$$
\sum_{i=1}^N
\bigl(g_\varepsilon(\cdot,\cdot,a_i)-g(\cdot,\cdot,a_i)\bigr)
\mathbf 1_{E_i}
\longrightarrow 0
\quad\text{in }L^p(Q_T),
\qquad \varepsilon\downarrow0.
$$

Next, setting
$
E^{N}:=\bigcup_{i>N}E_i,
$  we know  $E^{N}\downarrow\emptyset$ as $N\to\infty$.
Then
\begin{align*}
&\Big\|
\sum_{i>N}
\bigl(g_\varepsilon(\cdot,\cdot,a_i)-g(\cdot,\cdot,a_i)\bigr)
\mathbf 1_{E_i}
\Big\|_{L^p} \le
\|(\Phi_\varepsilon+\Phi)\mathbf 1_{E^{N}}\|_{L^p}
\\
&\le
\|(\Phi_\varepsilon-\Phi)\mathbf 1_{E^{N}}\|_{L^p}
+
2\|\Phi\mathbf 1_{E^{N}}\|_{L^p}
\\
& \le
\|\Phi_\varepsilon-\Phi\|_{L^p}
+
2\|\Phi\mathbf 1_{E^{N}}\|_{L^p} ,
\end{align*}
where we use  Assumption \ref{ass:main}.
Then, using \eqref{eq:b-f-eps},
 we obtain
$$
\lim_{N\to\infty}\limsup_{\varepsilon\downarrow0}
\Big\|
\sum_{i>N}
\bigl(g_\varepsilon(\cdot,\cdot,a_i)-g(\cdot,\cdot,a_i)\bigr)
\mathbf 1_{E_i}
\Big\|_{L^p}
=0.
$$
Combining the convergence of the finite sum with the above tail estimate, we
conclude that the assertion (i) holds true.

We now prove (ii). For simplicity, for $g=b,f$, we write
$$
\bar g_\varepsilon(t,x):=g_\varepsilon(t,x,\theta(t,x)),
\qquad
\bar g(t,x):=g(t,x,\theta(t,x)), \quad (t,x)\in Q_T.
$$
Moreover, by (i) and the stability estimate for singular SDEs, see
\cite[Theorem~1.2]{GaleatiLing2023}, we have a constant $C$ independent of  $\varepsilon$ such that,
\begin{equation}\label{Xe-to-X}
\mathbb E\Big[
\sup_{r\in[s,T]}
|X_r^{\varepsilon,\theta}-X_r^\theta|
\Big]\leq C\|\bar b_\varepsilon-\bar b\|_{L^p}
\longrightarrow0 ,  \mbox{ as }\varepsilon\to0 .
\end{equation}

Next, since
$C_c^\infty(Q_T)$ is dense in $L^p(Q_T)$, for any $\eta>0$, we may choose
${\varphi_\eta}\in C_c^\infty(Q_T)$ such that
$
\|\bar f-{\varphi_\eta}\|_{L^p}<\eta .
$
Then, by Krylov's estimate again and \eqref{Xe-to-X},
\begin{align*}
&\mathbb E\int_s^T
\big|\bar f(r,X_r^{\varepsilon,\theta})
-\bar f(r,X_r^\theta)\big|\,\md r
\\
&\le
\mathbb E\int_s^T\big(
\big|\bar f(r,X_r^{\varepsilon,\theta})
-{\varphi_\eta}(r,X_r^{\varepsilon,\theta})\big|
 + \big|{\varphi_\eta}(r,X_r^\theta)
-\bar f(r,X_r^\theta)\big|\big)
\,\md r
\\
&\quad+
\mathbb E\int_s^T
\big|{\varphi_\eta}(r,X_r^{\varepsilon,\theta})
-{\varphi_\eta}(r,X_r^\theta)\big| \,\md r
\\
&\le
C\|\bar f-{\varphi_\eta}\|_{L^p}
+
\|\nabla_x{\varphi_\eta}\|_\infty
\int_s^T
\mathbb E\big|X_r^{\varepsilon,\theta}-X_r^\theta\big|\,\md r
\\
 %
&\le
C\eta
+
C\|\nabla_x{\varphi_\eta}\|_\infty
\mathbb E\Big[
\sup_{r\in[s,T]}
|X_r^{\varepsilon,\theta}-X_r^\theta|
\Big].
\end{align*}
Hence,
\begin{equation*}
\begin{aligned}
& \big|J_{\varepsilon}(\theta;s,x)-J(\theta;s,x)\big|
=
\Big|\mathbb E\int_s^T
\big(\bar f_{\varepsilon}(r,X_r^{\varepsilon,\theta})
-\bar f(r,X_r^\theta)\big)\,\md r\Big|
\\
&\le
\mathbb E\int_s^T
\big|(\bar f_{\varepsilon}-\bar f)(r,X_r^{\varepsilon,\theta})\big|
\,\md r
+
\mathbb E\int_s^T
\big|\bar f(r,X_r^{\varepsilon,\theta})
-\bar f(r,X_r^\theta)\big|\,\md r
\\
&\le
C\|\bar f_\varepsilon-\bar f\|_{L^p}
+C\eta
+
C\|\nabla_x{\varphi_\eta}\|_\infty
\mathbb E\Big[
\sup_{r\in[s,T]}
|X_r^{\varepsilon,\theta}-X_r^\theta|
\Big],
\end{aligned}
\end{equation*}
We emphasize that the generic constant $C>0$ appearing
in the above estimates is independent of $\varepsilon$ and $\eta$. Its
uniformity in $\varepsilon$ follows from
$$
\|\bar b_\varepsilon\|_{L^p }
\leq
\|\Phi_\varepsilon\|_{L^p }
\leq
\|\Phi\|_{L^p },
\qquad
\|\bar b\|_{L^p }
\leq
\|\Phi\|_{L^p },
$$
and the corresponding uniform Krylov estimates.

 {For each fixed $\eta>0$, the quantity
$\|\nabla_x\varphi_\eta\|_\infty$ is finite and independent
of $\varepsilon$. Hence, letting $\varepsilon\downarrow0$ gives
\[
\limsup_{\varepsilon\downarrow0}
|J_\varepsilon(\theta;s,x)-J(\theta;s,x)|
\le C\eta.
\]
Letting $\eta\downarrow0$ proves {\rm(ii)}.}
\end{proof}

\medskip

\begin{proof}[Proof of Proposition~\ref{Prop:V>V_n}]
 Fix $(s,x)\in Q_T$ and let $\delta>0$ be arbitrary. To prove the
limsup inequality, it is enough to test the regularized value function by
feedback controls. Indeed, every Borel feedback selector induces an admissible
weak control system for the regularized problem, and therefore, for every
Borel feedback selector
$\theta:Q_T\to A$,
$
V_\varepsilon(s,x)\le J_\varepsilon(\theta;s,x).
$
By Theorem~\ref{thm:verification}, there exists a Borel measurable feedback
selector
$
\theta^\delta:Q_T\to A
$
such that the induced feedback control is admissible and
$$
J(\theta^\delta;s,x)\le V(s,x)+\delta .
$$
Applying Lemma~\ref{lem:composition-convergence}   to the
fixed feedback selector $\theta^\delta$, we get
$$
J_\varepsilon(\theta^\delta;s,x)
\longrightarrow
J(\theta^\delta;s,x),
\qquad \varepsilon\downarrow0 .
$$

On the other hand, the feedback control induced by $\theta^\delta$ is admissible
for the regularized problem, so the definition of the regularized value function gives
$$
V_\varepsilon(s,x)\le J_\varepsilon(\theta^\delta;s,x),
\qquad \varepsilon>0 .
$$
Therefore,
$$
\limsup_{\varepsilon\downarrow0}V_\varepsilon(s,x)
\le
\lim_{\varepsilon\downarrow0}J_\varepsilon(\theta^\delta;s,x)
=
J(\theta^\delta;s,x)
\le
V(s,x)+\delta .
$$
Since $\delta>0$ is arbitrary, we conclude that
$$
\limsup_{\varepsilon\downarrow0}V_\varepsilon(s,x)
\le V(s,x),
\qquad (s,x)\in Q_T .
$$
The proof is complete.
\end{proof}

\section*{Acknowledgments}
 The authors are grateful to the anonymous reviewers for their careful reading
and constructive comments, which have helped improve the quality and clarity of
the paper.

 \appendix
 \section{A weak comparison principle}

For related maximum estimates with singular drift, see
\cite[Theorems~1.3 and~3.6]{ZhangMP}.

\begin{lemma} \label{lem:whole-space-max}
   Suppose that \(p>d+2\) and
$
\theta\in L^p(Q_T;\mathbb R^d),
$ $
F\in L^p(Q_T).
$
Let \(w\in W_p^{1,2}(Q_T)\) satisfy
\begin{equation*}\label{eq:linear-ineq-terminal}
\partial_s w+\Delta w+\theta(s,x)\cdot\nabla w\ge -F(s,x)
\quad\text{in } Q_T
\end{equation*}
 in the weak sense: for every nonnegative
\(\phi\in C_c^{1,\infty}(Q_T)\),
\begin{equation*}\label{eq:linear-ineq-weak}
\iint_{Q_T}
\left[
-w\,\partial_s\phi
-\nabla w\cdot\nabla\phi
+\theta(s,x)\cdot\nabla w\,\phi
+F\phi
\right]\,\md x\md s
\ge0.
\end{equation*}
Assume moreover that
$
w(T,x)\le0,$ $ x\in\mathbb R^d .
$
Then
\begin{equation}\label{eq:linear-max-estimate}
\sup_{Q_T} w^+
\le
C\|F^+\|_{L^p },
\end{equation}
where \(C\) depends only on \(d,p,T\) and
\(\|\theta\|_{L^p }\). In particular, if \(F=0\), then
$
w\le0$ $\text{in }Q_T .
$
\end{lemma}

\begin{proof}
Fix $(s,x)$ with $s<T$. By Lemma~\ref{lem:state-equation-solvability}{\rm(ii)},
the equation $\md X_r=\theta(r,X_r)\,\md r+\sqrt{2}\,\md W_r$, $X_s=x$,
has a strong solution, taking a Borel representative of $\theta$.
Since $w\in W_p^{1,2}$ and $\nabla w$ is bounded, all terms in the
differential inequality belong to $L^p$ and the inequality holds almost
everywhere. The Krylov estimate allows us to evaluate it along $X$.
Applying the generalized It\^o formula as in Remark~\ref{rem:sobolev-ito},
taking expectations and using $w(T,\cdot)\le0$, we obtain
\[
 w(s,x)\le \mathbb E\,w(T,X_T)+\mathbb E\int_s^T F(r,X_r)\,\md r
 \le \mathbb E\int_s^T F^+(r,X_r)\,\md r\le C\|F^+\|_{L^p}.
\]
The last inequality follows from Lemma~\ref{lem:state-equation-solvability}
with envelope $|\theta|$, giving the stated dependence of $C$.
\end{proof}

We now prove the comparison principle for the HJB equation \eqref{eq:hjb}. For convenience, we set
$\displaystyle
\mathbb H(s,x,P)
:=
\inf_{a\in A}
\{b(s,x,a)\cdot P+f(s,x,a)\},
$ $  (s,x,P)\in Q_T\times \mathbb R^d .
$

\smallskip

\begin{lemma}
\label{lem:HJB-weak-comparison}
Under Assumption \ref{ass:main},
let \(u,v\in W_p^{1,2}(Q_T)\) be, respectively, a weak subsolution and a weak
supersolution of HJB equation \eqref{eq:hjb}, namely, for every nonnegative
\(\phi\in C_c^{1,\infty}(Q_T)\),
\[
\iint_{Q_T}
-\big(u\partial_s\phi+\nabla u\cdot\nabla\phi\big)\,\md x\md s
+\iint_{Q_T}\mathbb H(s,x,\nabla u)\phi\,\md x\md s
\ge0,
\]
and
\[
\iint_{Q_T}
-\big(v\partial_s\phi+\nabla v\cdot\nabla\phi\big)\,\md x\md s
+\iint_{Q_T}\mathbb H(s,x,\nabla v)\phi\,\md x\md s
\le0.
\]
Assume moreover that
$
u(T,x)\le v(T,x),$ $ x\in\mathbb R^d .
$
Then
$
u(s,x)\le v(s,x),$ $ (s,x)\in Q_T .
$
\end{lemma}

\begin{proof}
Setting
$
W:=u-v,
$ we obtain, for every nonnegative
\(\phi\in C_c^{1,\infty}(Q_T)\),
\begin{equation}\label{eq:diff-weak-HJB}
\begin{aligned}
&0\le \iint_{Q_T}
-\big(W\partial_s\phi+\nabla W\cdot\nabla\phi\big)\,\md x\md s  \\
&\qquad
+\iint_{Q_T}
\big(\mathbb H(s,x,\nabla u)-\mathbb H(s,x,\nabla v)\big)\phi\,\md x\md s\\
 &= \iint_{Q_T}
-\big(W\partial_s\phi+\nabla W\cdot\nabla\phi\big)\,\md x\md s
+\iint_{Q_T}
\theta(s,x)\cdot\nabla W(s,x)\phi\,\md x\md s,
\end{aligned}
\end{equation}
where \(\theta:Q_T\to\mathbb R^d\) is introduced as follows,
\[
\theta(s,x):=
\begin{cases}
\dfrac{\mathbb H(s,x,\nabla u)-\mathbb H(s,x,\nabla v)}
      {|\nabla u-\nabla v|^2}
      \big(\nabla u-\nabla v\big),
& \nabla u(s,x)\ne\nabla v(s,x),\\[2ex]
0,
& \nabla u(s,x)=\nabla v(s,x).
\end{cases}
\]
Notice that, for any \(P_1,P_2\in\mathbb R^d\),
\[
\begin{aligned}
|\mathbb H(s,x,P_1)-\mathbb H(s,x,P_2)|
&\le
\sup_{a\in A}
|b(s,x,a)\cdot(P_1-P_2)| \le
\Phi(s,x)|P_1-P_2|,
\end{aligned}
\]
for a.e. \((s,x)\in Q_T\).
We know $
|\theta(s,x)|\le \Phi(s,x),$ and
$
\theta\in L^p(Q_T;\mathbb R^d).
$ Hence \eqref{eq:diff-weak-HJB} shows that \(W\) satisfies the hypothesis of
Lemma~\ref{lem:whole-space-max} with \(F=0\). Since
$
W(T,\cdot) \le0$,  we get
$
W\le0 $ $\text{in }Q_T .
$
That is,
\[
u(s,x)\le v(s,x),\qquad (s,x)\in Q_T .
\]
The proof is complete.
\end{proof}

{
\section{Proof of Lemma \ref{lem:state-equation-solvability} and countexamples} \label{app:state-details}

\begin{proof}[Proof of Lemma~\ref{lem:state-equation-solvability}]
For part~{\rm(i)}, fix a solution $X$ on $[s,T]$ with $s<T$, and write
$a_t=B(t,X_t,\beta_t)$ and $D=(s,T)\times\mathbb R^d$.
To estimate the time $X$ spends in subsets of $D$, we use the backward
heat equation. Given $\lambda\ge1$ and $g\in C_c^\infty(D)$, set
\[
w(t,\cdot)=\int_t^T e^{-\lambda(r-t)}e^{(r-t)\Delta}g(r,\cdot)\,\md r.
\]
Then $(\partial_t+\Delta-\lambda)w=-g$ and $w(T,\cdot)=0$.
The heat-kernel bounds
$\|\nabla^j e^{u\Delta}v\|_\infty
\le C u^{-j/2-d/(2p)}\|v\|_{L^p(\mathbb R^d)}$ for $j=0,1$,
followed by H\"older's inequality in time, give
\[
2\|w\|_\infty+\|\nabla w\|_\infty
\le C\lambda^{-\kappa}\|g\|_{L^p(D)},\qquad
\kappa=\tfrac12-\tfrac{d+2}{2p}>0,
\]
where $C$ depends only on $d,p$.

To apply this bound along $X$, we localize the drift by introducing the stopping time
\[
\tau_n=\inf\Big\{t\in[s,T]:\int_s^t|a_r|\,\md r\ge n\Big\}\wedge T .
\]
By the definition of a solution,
$\int_s^T|a_t|\,\md t<\infty$ a.s., and hence
$\tau_n\uparrow T$ a.s. Moreover, $D_n:=\mathbb E\int_s^{\tau_n}e^{-\lambda(t-s)}|a_t|\,\md t\le n$.
Define the discounted occupation measure $\nu_n$
on $D$ by
$
\nu_n(g)
:=
\mathbb E\int_s^{\tau_n}
e^{-\lambda(t-s)}g(t,X_t)\,\md t,
$
for every nonnegative Borel function $g$ on $D$.
Then $\nu_n$ is a finite positive Borel measure.

For any $g\in C_c^\infty(D)$, applying
It\^o's formula to
$e^{-\lambda(t-s)}w(t,X_t)$ up to $\tau_n$
and taking expectations yields
$$
|\nu_n(g)|
 \le 2\|w\|_\infty
     +D_n\|\nabla w\|_\infty\le C\lambda^{-\kappa}(1+D_n)
     \|g\|_{L^p(D)}.
$$
By density, $\nu_n$ extends to a bounded positive linear functional
on $L^p(D)$. Hence,  writing $p'=p/(p-1)$, there exists
 $\rho_n\in L^{p'}(D)$, with $\rho_n\ge0$, such that
\[
\nu_n(g)=\int_D g(t,y)\rho_n(t,y) \md t\md y,
\quad g\in  L^p(D),\quad
\|\rho_n\|_{L^{p'}(D)}
\le C\lambda^{-\kappa}(1+D_n).
\]
Since $\nu_n$ is a finite positive Borel measure and
$C_c^\infty(D)$ determines locally finite Borel measures, we have
$
\nu_n(\md t,\md y)=\rho_n(t,y)\,\md t\md y.
$
Consequently, for every nonnegative Borel function $g\in L^p(D)$,
\[
\nu_n(g)
\le \|\rho_n\|_{L^{p'}(D)}\|g\|_{L^p(D)}
\le C\lambda^{-\kappa}(1+D_n)\|g\|_{L^p(D)}.
\]

We may now take $g=\Psi$. The drift bound $|a_t|\le\Psi(t,X_t)$ gives
$D_n\le\nu_n(\Psi)\le C\lambda^{-\kappa}(1+D_n)\|\Psi\|_{L^p(Q_T)}$.
Choosing $\lambda$ so that $C\lambda^{-\kappa}\|\Psi\|_{L^p(Q_T)}\le1/2$
gives $D_n\le1$ for every $n$.
Thus,
$
\nu_n(|h|)
\le
2C\lambda^{-\kappa}\|h\|_{L^p(Q_T)}
$
for every Borel $h\in L^p(Q_T)$.
Finally, since $\tau_n\uparrow T$, a.s., monotone
convergence gives
\[
\begin{aligned}
&\mathbb E\int_s^T|h(t,X_t)|\,\md t
\le
e^{\lambda T}
\mathbb E\int_s^T
e^{-\lambda(t-s)}|h(t,X_t)|\,\md t\\
&=
e^{\lambda T}\lim_{n\to\infty}\nu_n(|h|)
\le
2C_0e^{\lambda T}\lambda^{-\kappa}
\|h\|_{L^p(Q_T)}.
\end{aligned}
\]
 So, \eqref{eq:krylov-weak} is proved.

For part~{\rm(ii)}, composing $B$ with the feedback law $\theta$ gives a
deterministic Borel drift $(t,y)\mapsto B(t,y,\theta(t,y))$ with $L^p(Q_T)$
norm at most $\|\Psi\|_{L^p(Q_T)}$. Since $(d+2)/p<1$, strong existence and
uniqueness follow from \cite[Theorem~1.1]{Zhang2011}. Applying
part~{\rm(i)} to this solution gives the Krylov estimate.
\end{proof}

\medskip

 \hypertarget{app:state-examples}{}The following examples illustrate the assertions of Remark~\ref{Exam:1.2}. In both, take $d=1$, $E=\mathbb R$,
$X_0=0$, and prescribe $\beta_t=\sqrt2 W_t$.
Both drifts are bounded by $\varphi(x)=e^{-x^2}$,
which belongs to $L^p(Q_T)$ for every $1\le p<\infty$.

{\color{black}
\begin{example}[Failure of weak uniqueness]\label{ex:open-loop-nonuniqueness}
Take $B(t,x,e)=\varphi(x){\bf 1}_{\{x>e\}}$. The state equation becomes
\[
\md X_t=\varphi(X_t){\bf 1}_{\{X_t>\sqrt2W_t\}}\,\md t+\sqrt2\,\md W_t .
\]
One solution is \(X_t=\sqrt2W_t\), along which the drift vanishes.
To construct a second solution, fix \(\tau\in[0,T)\) and set \(Y_t=0\) for \(t\le\tau\). For \(t>\tau\), let \(Y\) solve
\[
\dot Y_t=\varphi(Y_t+\sqrt2W_t),\qquad Y_\tau=0 .
\]
Then \(X_t=\sqrt2W_t+Y_t\) is another solution driven by the same \(W\) and the
same control \(\beta\). For $t>\tau$, $0<Y_t\le t-\tau$, so the two
solutions have different expectations and hence different laws.
\end{example}
}

\begin{example}[Failure of weak existence]\label{ex:open-loop-nonexistence}
Keep the same control and replace the drift by
$\widehat B(t,x,e)=\varphi(x){\bf1}_{\{x=e\}}$. If a solution existed on any stochastic basis
preserving the relation $\beta_t=\sqrt2 W_t$, then $Z_t=X_t-\sqrt2 W_t$ would satisfy
\[
Z_t=\int_0^t e^{-2W_r^2}{\bf1}_{\{Z_r=0\}}\,\md r .
\]
Thus $Z$ is nonnegative and nondecreasing. If $Z_t=0$ for some $t>0$,
then $Z=0$ on $[0,t]$, so the integral is strictly positive.
Hence $Z_t>0$ for every $t>0$, which makes the integral identically zero.
This contradiction rules out a weak solution for the prescribed control.
\end{example}
}

\bibliographystyle{siam}
\bibliography{ref}

\end{document}